\documentclass[11pt]{article}

\usepackage{hyperref}

\usepackage{amsfonts,amsmath,amssymb,amsthm}

\usepackage{hhline}
\usepackage[english]{babel}
\usepackage[utf8]{inputenc}
\usepackage{url}
\usepackage{graphicx}
\usepackage[T1]{fontenc}
\usepackage{textcomp}
\usepackage{floatflt}
\usepackage[font={footnotesize}]{caption}
\usepackage{subcaption}
\usepackage[dvipsnames]{xcolor}
\usepackage{makecell}

\definecolor{szin}{rgb}{0,0.44,0.4}
\definecolor{szin2}{rgb}{0.902,0.2705,0}
\definecolor{szin3}{rgb}{0,0.5,0}

\hypersetup{pdfpagemode=UseNone, pdftoolbar=true, breaklinks=true,
colorlinks=true,
linkcolor=szin2,
linktoc=page,
citecolor=szin3,bookmarks=false}
\makeindex

\def\E{\mathbb{E}}
\def\Pr{\mathbb{P}}
\def\Var{\mbox{\rm Var}}

\newcommand{\restr}[1]{\!\!\upharpoonright_{#1}}

\newcommand{\eps}{{\mbox{$\epsilon$}}}
\newcommand{\R}{{\mathbb R}}

\newcommand{\Z}{{\mathbb Z}}
\newcommand{\N}{{\mathbb N}}
\newcommand{\T}{{\overline{T}}}
\newcommand{\F}{{\mathcal F}}

\newcommand{\CC}{{\mathcal C}}
\newcommand{\A}{{\mathcal A}}
\newcommand{\B}{{\mathcal B}}

\newcommand{\RR}{{\mathcal R}}

\newcommand{\GC}{{\mathcal G}}

\newcommand{\TT}{\mathcal{T}}

\newcommand{\Past}{\mathsf{Past}}
\newcommand{\Future}{\mathsf{Future}}
\newcommand{\Step}{\mathsf{Step}}

\def\eps{\epsilon}

\def\1{\mathbf{1}}
\def\aa{\alpha}

\usepackage{mathrsfs}

\def\P{\Pr}
\def\Pn{\Pr_{N}}
\def\Pinf{\Pr_{\infty}}
\def\md{\mid}
\def\Bb#1#2{{\def\md{\bigm| }#1\bigl(#2\bigr)}}
\def\BB#1#2{{\def\md{\Bigm| }#1\Bigl(#2\Bigr)}}
\def\Bs#1#2{{\def\md{\mid}#1(#2)}}
\def\Pb{\Bb\P}
\def\Eb{\Bb\E}
\def\PB{\BB\P}
\def\PnB{\BB\Pn}
\def\EB{\BB\E}
\def\Ps{\Bs\P}
\def\Pns{\Bs\Pn}

\def\Es{\Bs\E}
\DeclareMathOperator{\sign}{sgn}

\def\UST{\mathsf{UST}}
\def\ST{\mathsf{St}}

\def\pow{\mathsf{pow}}
\def\shiftpow{\mathsf{shiftpow}}

\newcommand{\avg}[1]{\left< #1 \right>} 

\newtheorem{theorem}{Theorem}
\numberwithin{theorem}{section}

\newtheorem{lemma}[theorem]{Lemma}
\newtheorem{proposition}[theorem]{Proposition}

\newtheorem{corollary}[theorem]{Corollary}

\numberwithin{equation}{section}
\numberwithin{figure}{section}

\graphicspath{{./plots/} }

\begin{document}


\title{Speeding up non-Markovian First Passage Percolation\\ with a few extra edges}

\author{Alexey Medvedev 
\thanks{Namur Institute for Complex Networks (naXys), Universit\'e de Namur, Rempart de la Vierge, 8, Namur 5000 Belgium and ICTEAM, Universit\'e Catholique de Louvain, Av. Georges Lemaitre 4, 1348 Ottignies-Louvain-la-Neuve, Belgium. Email: {\tt an\_medvedev `at' yahoo `dot' com}}
\and
G\'abor Pete
\thanks{Alfr\'ed R\'enyi Institute of Mathematics, Re\'altanoda u.~13-15., Budapest 1053 Hungary, and Budapest University of Technology and Economics, Egry J.~u.~1, Budapest 1111 Hungary. Email: {\tt robagetep `at' gmail `dot' com}}
}

\date{}

\maketitle

\begin{abstract}
One model of real-life spreading processes is First Passage Percolation (also called SI model) on random graphs. Social interactions often follow bursty patterns, which are usually modelled with i.i.d.~heavy-tailed passage times on edges. On the other hand, random graphs are often locally tree-like, and spreading on trees with leaves might be very slow, because of bottleneck edges with huge passage times. Here we consider the SI model with passage times following a power law distribution $\Ps{\xi>t}\sim t^{-\aa}$, with infinite mean. For any finite connected graph $G$ with a root $s$, we find the largest number of vertices $\kappa(G,s)$ that are infected in finite expected time, and prove that for every $k \leq \kappa(G,s)$, the expected time to infect $k$ vertices is at most $O(k^{1/\alpha})$. Then, we show that adding a single edge from $s$ to a random vertex in a random tree $\TT$ typically increases $\kappa(\TT,s)$ from a bounded variable to a fraction of the size of $\TT$, thus severely accelerating the process. We examine this acceleration effect on some natural models of random graphs: critical Galton-Watson trees conditioned to be large, uniform spanning trees of the complete graph, and on the largest cluster of near-critical Erd{\H o}s-R\'enyi graphs. In particular, at the upper end of the critical window, the process is already much faster than exactly at criticality. 
\end{abstract}

\section{Introduction}\label{s.intro}

\subsection{Motivation and background}

Spreading is one of the most important dynamic processes on complex networks as it is the basis of a broad range of phenomena from epidemic contagion to diffusion of innovations \cite{Vespignani2012, RevModPhys2015}. One of the original, and still primary, reasons for studying networks is to understand the mechanisms by which diseases, information, computer viruses, rumors, innovations spread over them \cite{MN1}. The spreading problem initially came from epidemiology, and the terminology still remembers these roots: in the simplest spreading model, the two-state \textit{susceptible-infected (SI) model}, any particular person can be either in susceptible (S) or in infected (I) state, with the transition rule $\text{S}\to\text{I}$, meaning that once an infection is obtained, the person stays infected forever. It is usually assumed that initially there is only one infected person.

In the usual mathematical model of SI spreading, a rooted connected simple graph $G=(V,E,s)$ is given, with vertices representing people, and edges representing the connections between them, through which the infection can pass. Initially, only the root $s\in V$ is infected.  The edges $e\in E$ have i.i.d.~random positive weights $\xi(e)$, with common distribution $\xi$, representing the passage time of the infection. Infection is then transmitted along the edges from infected vertices to susceptible ones with passage times according to the edge weights (measured from the moment of infecting the first endpoint of the edge). The process runs until all vertices get infected. Alternatively, and this is how the model was first defined in the mathematical community under the name {\it First Passage Percolation (FPP)} \cite{Hammersley}, one may consider $\xi(e)$ as the length of the edge $e\in E$, and then we can think of some fluid flowing through this random medium at speed 1, starting from the source $s$. In the two models, SI and FPP, the moment when a vertex is reached by the infection/fluid will be the same. Then, we can measure spreading via two different, but of course closely related stochastic processes: either by $(T_k)_{k=1}^{|V|}$, the sequence of times when exactly $k$ vertices are infected in the process, or by $(N_t)_{t\geqslant 0}$, the number of vertices reached by time $t$. In FPP one is typically more concerned with the process $(N_t)_{t\geqslant 0}$ (see, e.g., \cite{Kesten1987}), while in the current paper we consider the process $(T_k)_{k=1}^{|V|}$, therefore we will prefer the SI model notation.

The SI spreading model has attracted much attention in the network science community, giving rise to a solid number of papers  (e.g., \cite{IribarrenMoro2009, Karsai_etal2011, HolmeMasuda2013, Jo_etal2014, HorvathKertesz2014}). In the mathematical literature, there has been an enormous number of papers studying this model on various graphs and using various distributions for $\xi$; see, e.g., \cite{Hammersley, Kesten1987, FillPemantle1993, BvdHH, vdHBhamidi2011, Komjathy2014, bhamidi2014universality}; two surveys are \cite{FPP50} on the first 50 years of the model on Euclidean lattices and \cite[Chapter 8]{RemcoV2} on other graphs. Although originally $\xi$ was typically assumed to have Exponential or Bernoulli distribution, more recently, motivated by the bursty nature of communications in many real networks, $\xi$ with {\it power law distributions} $\pow(\aa)$ have also been considered: $\Ps{\xi>t}=1\wedge(t/t_0)^{-\aa}$, where $t_0>0$ and $\aa>0$. In this case the model is non-Markovian, which creates many obstacles in studying it; this will be the situation in the present work, as well. 

Most sparse (e.g., bounded average degree) random graph models produce graphs that are {\it locally tree-like}: a large neighbourhood of a random root is a tree with high probability, maybe with some extra edges decorating the tree. Typically, {\it supercritical} random graph models have been studied, where the unique giant connected component looks locally like a fast-growing supercritical Galton-Watson tree, maybe with decorations. (In more precise terms, the Benjamini-Schramm limit of the giant clusters is an infinite unimodular random tree with average degree larger than 2, or more generally, an invariantly non-amenable unimodular random graph \cite{BSch,AldousLyons}). This local fast-growing tree structure has been used in several works very successfully to understand the behavior of SI and FPP on the entire graph \cite{Komjathy2014, bhamidi2014universality, BaroniKomjathy2017}, even with a general absolutely continuous distribution for $\xi$.

However, it was observed in \cite{IribarrenMoro2009} that the SI spreading with heavy-tailed $\xi$ on {\it finite trees}, when started from a typical site, may be very slow. Real-life examples with such $\xi$ include, for example, information cascades in Twitter \cite{YeraliWiki2017, MartinBivas2017}. A striking feature of slow spreading when $\xi$ has infinite mean was noticed in \cite{HorvathKertesz2014} via computer simulations, which is most apparent when not the curve $t\mapsto \langle N_t\rangle$ is considered, but the ``inverse'' curve $k \mapsto \langle T_k \rangle$. Namely, running the simulation of the process $T=(T_k)_{k=1}^{n}$ on a tree for $M$ times, and definining $\avg{T_k}=\frac{1}{M}\sum_{i=1}^{M}T^{(i)}_k$, the \textit{spreading curve} $\left(\avg{T_k}, k/n\right)$ exhibits ``uncontrolled large plateaux'', which do not decrease and do not converge as we  increase the number of runs (see the dark blue curves on Figure~\ref{fig:tree_vs_cycle}~(a)). 
Plateaux of similar type were also empirically found in \cite{MartinBivas2017}. On the other hand, adding just one extra edge to the tree ({\it which does not change the local statistics of the graph!}) makes the spreading curve quite smooth; see the lighter green curves on Figure~\ref{fig:tree_vs_cycle}~(a), and also Figure \ref{fig:tree_vs_cycle}~(b) that shows SI spreading on the cycle, with power law exponent $\aa\in (1/2,1)$.

\begin{figure}[h]
\centering
\begin{minipage}[b]{0.49\textwidth}
\centering
\includegraphics[width=0.8\linewidth]{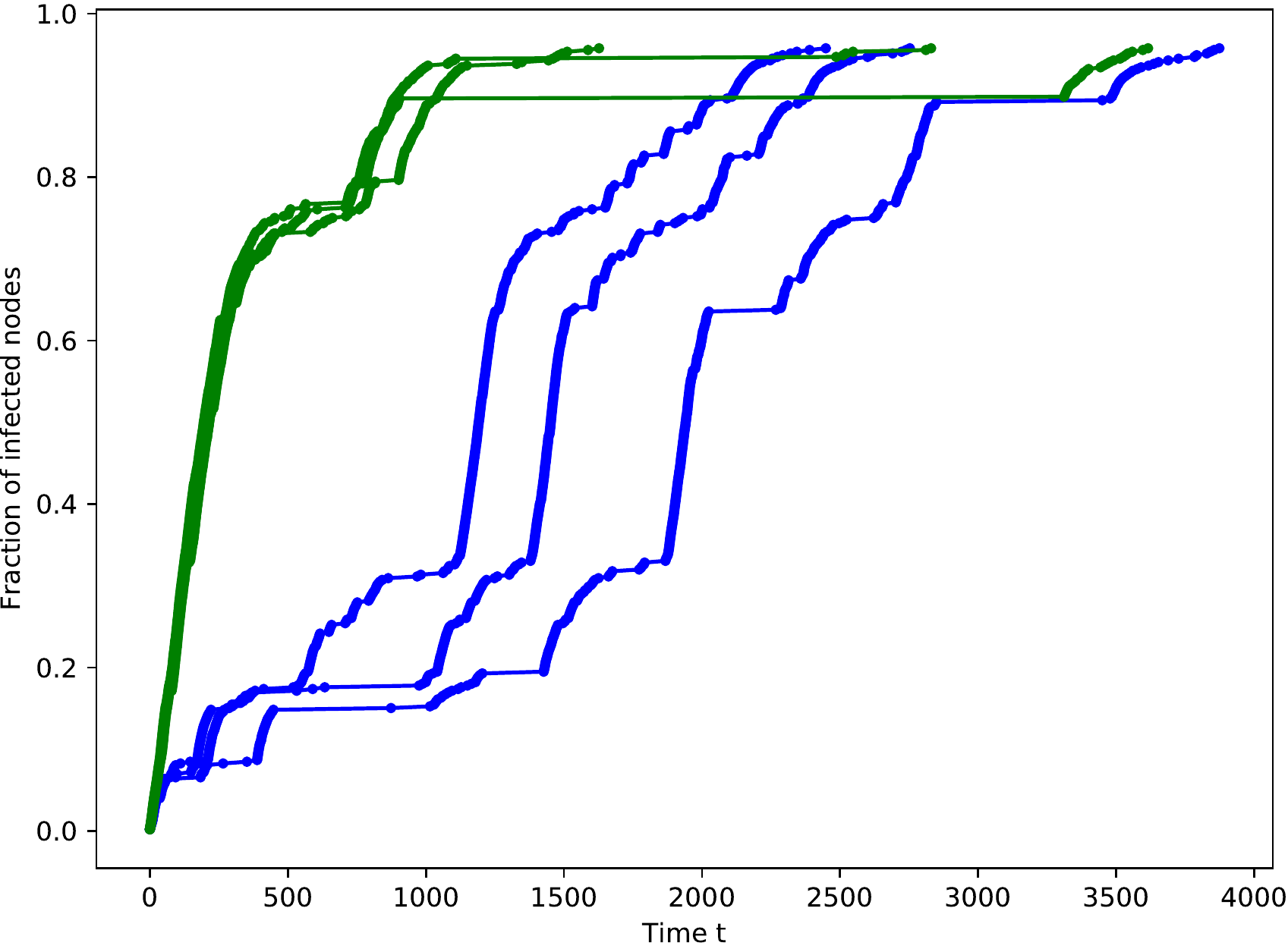}
\caption*{(a)}
\end{minipage}
\begin{minipage}[b]{0.49\textwidth}
\centering
\includegraphics[width=0.8\linewidth]{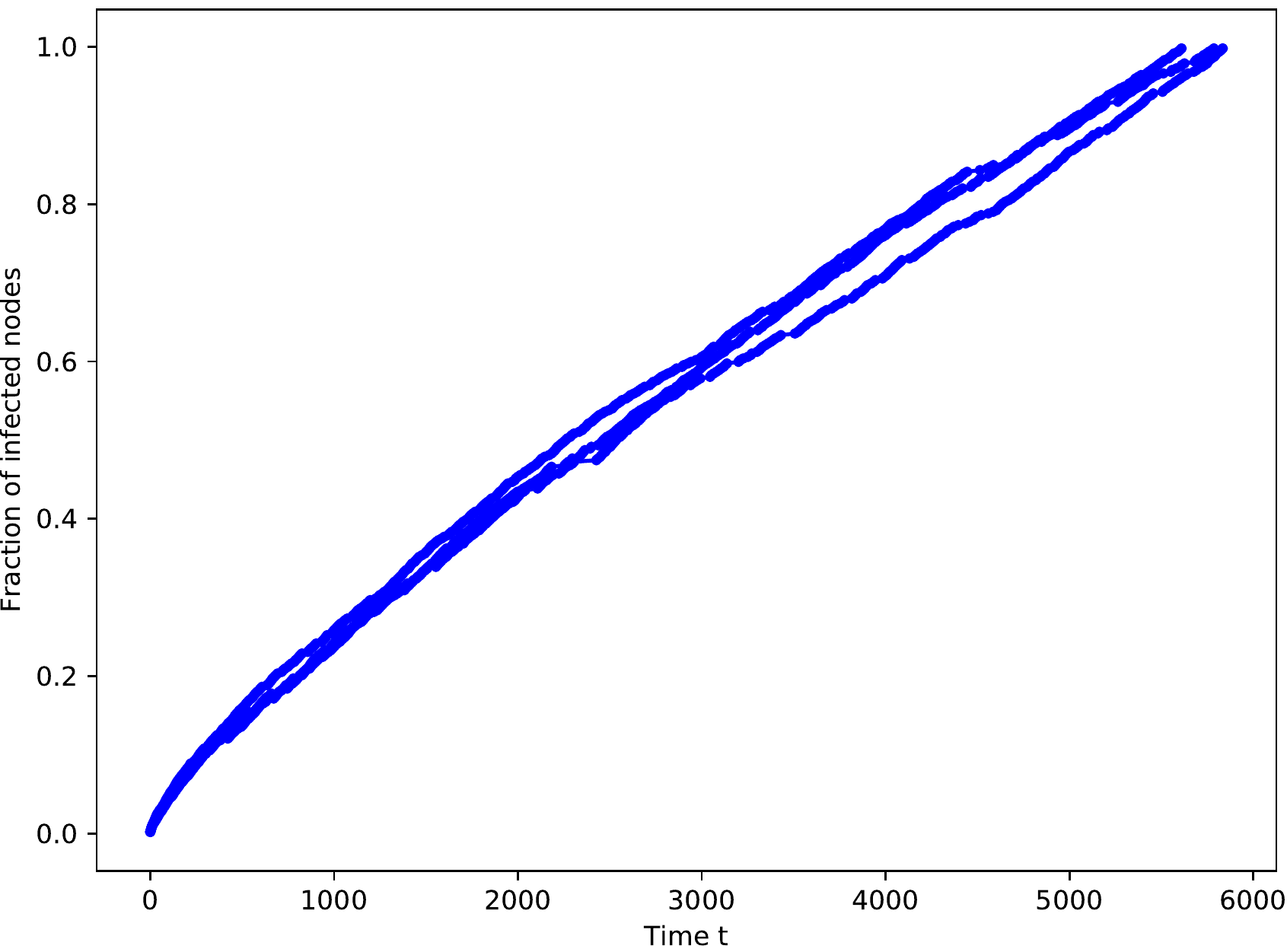}
\caption*{(b)}
\end{minipage}
\caption{Spreading curves of the SI model simulation with power law weights $\xi$ with $\aa=0.8$. (a) For the three dark \textcolor{blue}{blue} curves, the underlying graph is a tree with $472$ vertices; for the three lighter \textcolor{ForestGreen}{green} curves, one fixed edge is added between the root and a random vertex; b) the underlying graph is a cycle $C_n$ with $472$ vertices.}
\label{fig:tree_vs_cycle}
\end{figure}

The first motivation for our paper was to understand the striking effect of adding a few edges to large finite trees. After describing this phenomenon in detail for arbitrary deterministic graphs, we address some natural models of random trees, using a variety of techniques. Since supercritical trees are typically not finite (or, if a supercritical GW tree is conditioned to be finite, it becomes subcritical \cite[Section 5.7]{LPbook}), the most natural and interesting case seems to be adding a random edge to a critical random tree, conditioned to be large in one way or another. Or, in the most natural model that we will consider, the largest cluster in a {\it near-critical Erd{\H o}s-R\'enyi graph} $G\big(n,\frac{1}{n}+\frac{\lambda}{n^{4/3}}\big)$, for $\lambda\ll 0$ is very likely to be a large critical tree, while for $\lambda\gg 0$ is very likely to be a critical tree with several extra edges (see Section~\ref{s.ER} for references). In particular, the SI process with power law weights for $\lambda\ll 0$ is slow and has uncontrolled large plateaux, while most of the graph is infected in a fast and smooth way for $\lambda\gg 0$ (see Figure~\ref{fig:ER_simulation}).

The plateaux are explained by the presence of so-called temporal ``bottlenecks'', which are large passage times that occur on some particular edges. Indeed, if there is a $k$ for which $\Es{T_k}<\infty$ but $\Es{T_{k+1}}=\infty$, then the Law of Large Numbers tells us that 
$$\frac{1}{M}\left(\sum_{i=1}^{M}T^{(i)}_{k+1} - \sum_{i=1}^{M}T^{(i)}_k\right) \to \infty \text{ a.s., as }M\to\infty,$$
and hence the first uncontrolled plateaux in the spreading curve is between $\avg{T_k}$ and $\avg{T_{k+1}}$. The next plateau occurs as if we started the process anew after the first bottleneck, without the already infected part of the graph. Hence all the later plateaux can be investigated with the method we describe below, and here we will focus on the time when the first plateau appears.

It is easy to obtain some intuition how the temporal bottlenecks disappear for the SI process on the cycle with $\aa\in(1/2,1)$:  there are always at least \textit{two} edges through which the infection might proceed, and the minimum of two independent $\pow(\alpha)$ variables already has a finite expectation. In fact, we will show in Subsection~\ref{ss.spreading_Z} that $\Es{T_k}\asymp k^{1/\aa}<\infty$ for every $1\leqslant k\leqslant n$. However, because of the non-Markovian nature of the process (with already old edges getting occupied even later), making such an argument rigorous for a general graph with cycles is not at all trivial. 

In any case, we see that for cycles $\aa=1/2$ must be a threshold, and in the current paper we will concentrate on the case of $\aa\in(1/2,1)$. Nevertheless, our results can be extended to the case when the infection has $d>2$ ways to proceed and $\aa>1/d$;  simplistically, the main reason for this bound is that then the minimum of $d$ independent $\pow(\alpha)$ variables has a finite expectation. See Subsection~\ref{ss.extopen} for more details.

\subsection{Results and outline}

Summarizing the above assumptions, in this paper we consider the SI spreading model with i.i.d. power law transmission times and we study the role of cycles in speeding up of the process when $\aa\in (1/2,1)$ and thus eliminating large plateaux on (or `smooth') the average spreading curve up to a certain level. We rigorously identify when the first possible temporal bottleneck appears and its relation to the graph structure. We show that for each finite connected graph $G$ on $n$ vertices there exists a specific threshold $\kappa(G,s)$, where $s$ is an initially infected vertex, such that the average time to infect $k\leqslant\kappa(G,s)$ vertices is finite, and is even bounded by a polynomial function in $k$, and when $k> \kappa(G,s)$ the average time is infinite. The result, proved in Section~\ref{s.determ}, can be stated as follows.

\begin{theorem}\label{thm:main_smoothing}
Consider a graph $G$ on $n$ vertices with a root $s$ and the SI spreading process $T=\left(T_j\right)_{j=1}^{n}$ with power law weights with $\aa\in(1/2,1)$. Then there exists the number $\kappa(G,s)$ such that for each $k$, where $1\leqslant k\leqslant \kappa(G,s)$, the expected time to infect $k$ vertices is bounded by
$$\Es{T_k}\leqslant C k^{1/\aa},$$
and for $k> \kappa(G,s)$, the expectation $\Es{T_k}=\infty$.
\end{theorem}

The number $\kappa(G,s)$ identifies the position of the first temporal bottleneck for the process, and has a simple combinatorial description, which is presented in the following Lemma. 

\begin{lemma}\label{lem:comb_kappa}
Let $G$ be a finite rooted graph with root $s$ and let $T$ be the $SI$ spreading process on $G$ with weights having absolutely continuous distribution $\xi$, such that $\Es{\xi}=\infty$. Then,
$$\kappa(G,s) = \min_{e\in E(G)}|\CC(s,G\backslash e)|,$$
where $|\CC(s,G\backslash e)|$ is the size of the connected component of vertex $s$ in the graph $G$ without edge $e$.
\end{lemma}

In Section~\ref{s.two}, we will describe two examples, the cycle and the star graph, for both of which we have $\kappa=n-1$ and $\Es{T_\kappa}\asymp n^{1/\alpha}$, but {\it for quite different reasons}. We have been unable to identify either of these examples as the ``hardest graph to infect'', which can be considered as an explanation for the non-triviality of Theorem~\ref{thm:main_smoothing}.
\medskip

Next, we will study the value of $\kappa(G,s)$ in some near-critical random graph models, and show that by adding one or a few edges, it increases from a bounded value to a positive proportion of all the vertices. In our first example, we add an extra edge to a {\it critical Galton-Watson tree} conditioned to have depth at least $N$, between the starting vertex $s$ and a uniform random vertex. Here our main tool is comparison with the critical GW tree conditioned to be infinite \cite{Kesten1986}, and our goal is to be as elementary as possible, to develop arguments that may be applicable to other critical random tree models that one may be interested in (see Section~\ref{s.GW}). In our second example, we add one random edge to the {\it uniform spanning tree} of the complete graph; this model makes it possible to use elegant methods relying on Loop-Erased Random Walks (via Wilson's algorithm \cite{Wilson1996}) and P\'olya urns with time-dependent increments \cite{Pemantle1990}, very different from the first example (see Section~\ref{s.UST}). Finally, our third example is the ``physically'' most relevant one: the {\it near-critical Erd{\H o}s-R\'enyi graph} $G\big(n,\frac{1}{n}+\frac{\lambda}{n^{4/3}}\big)$, $\lambda\in\R$. Here the analysis is based on fine understanding of the metric structure of the graph \cite{Addario-Berry2010-2}, encoded typically as a Gromov-Hausdorff scaling limit using Aldous' Continuum Random Tree \cite{AldousCRTII} as a basic building block (see Section~\ref{s.ER}).

In the first two of these models, in order to raise $\kappa(G,s)$ significantly, it is important that the extra edge is incident to $s$, or in other words, that $s$ lies in the cycle that we have created. Instead of adding the extra edge, we could have moved $s$ onto the ``spine'' of the tree, so that there are two disjoint long paths emanating from $s$, and this would have a very similar accelerating and smoothing effect. So, in these cases, adding the extra edge could just be considered as a simple way to make the starting vertex special. This also makes sense from a practical point of view: if one's goal is to speed up the infection process, it is easier to create a new link, or to start the infection at two random places (which is almost the same), than finding a special vertex in the graph to start the process. In the third case, near-critical Erd{\H o}s-R\'enyi graphs, the starting vertex will not be special, and therefore $\kappa(G,s)$ will remain of constant order. However, after infecting a tiny portion of the graph, progressing with $k$ a little bit (although typically with a huge $T_k$), we get to a vertex with a large $\kappa$, and from there we infect most of the graph in a fast and smooth way. 
\medskip

We now state our exact results for the three near-critical random graph models. In our first example, we add a single extra edge $e$ between the initially infected vertex $s$ and a uniform random vertex in a critical Galton-Watson tree $\TT$. This graph will be denoted by $\TT_{+e}$, while $|\TT|$ denotes the number of vertices in $\TT$.

\begin{theorem}\label{thm:GW_smoothing}
Let $\TT^N$ be a critical Galton-Watson tree, conditioned on surviving to at least $N$ generations, $Z_N>0$, and consider the SI spreading process with power law weights with $\aa\in(1/2,1)$.  Then as $N\to\infty$,
\begin{enumerate}
\item[{\bf (1)}] the sequence of r.v. $\kappa(\TT^N,s)$ is tight;
\item[{\bf (2)}] for any $\varepsilon>0$ there exists $\delta>0$, such that
$$\PB{\frac{\kappa(\TT^N_{+e},s)}{|\TT^N_{+e}|}>\delta}>1-\varepsilon.$$
\end{enumerate}
\end{theorem}

Secondly, a similar result, but with a different approach, is provided for the model of uniform spanning tree of the complete graph.

\begin{theorem}\label{thm:UST}
Let $\UST(n)$ be the uniform spanning tree of the complete graph on $n$ vertices, and consider the SI spreading process with power law weights with $\aa\in(1/2,1)$, with starting vertex $s$. Let $e$ be a uniform random extra edge with one endpoint being $s$. Then, $\forall\varepsilon>0$ there is $\delta>0$ such that, for all $n$ large enough,
$$\PB{\frac{\kappa(\UST(n)_{+e},s)}{n}>\delta}>1-\varepsilon.$$
\end{theorem}

Thirdly:

\begin{theorem}\label{thm:ER}
Let $\CC=\CC^n_1(\lambda)$ be the largest cluster of the Erd{\H o}s-R\'enyi graph $G\big(n,\frac{1}{n}+\frac{\lambda}{n^{4/3}}\big)$.
\begin{itemize}
\item[{\bf (1)}] Let $\sigma$ be a uniform random vertex in $\CC$. Then, as $n\to\infty$, the sequence of variables $\kappa(\CC,\sigma)$ is tight for any fixed $\lambda\in\R$: for every $\eps>0$, if $K>0$ is large enough, then 
$$
\Pb{ \kappa(\CC,\sigma) < K }>1-\eps,
$$ for all $n>n_0(\lambda,\eps)$ large enough. 

\item[{\bf (2)}] For every $\lambda\in\R$ and $\eps>0$, there exists a $\delta_0(\eps)>0$ that is independent of $\lambda$, and a $\delta_1(\lambda,\eps)>0$ such that 
$$
\PB{\delta_0 |\CC| < \max_{s\in V(\CC)} \kappa(\CC,s) < (1-\delta_1) {|\CC|} } > 1-\eps\,.
$$
In particular, the spreading curve may have fast smooth parts for a positive proportion of the volume, but the maximum proportion is bounded away from 1.
\item[{\bf (3)}]  For every $\eps>0$, if $\lambda > 0$ is large enough, then
$$
\PB{\forall x\in V(\CC) \ \exists  s\in V(\CC) \text{ with }  k_x(s)<\eps |\CC| \text{ and } \kappa(\CC'_x,s) / |\CC| > 1-\eps} > 1-\eps\,,
$$
where $k_x(s)$ is the random number of steps in which the infection from $x$ reaches $s$, and $\CC'_x$ is the graph we obtain from $\CC$ by removing the vertices infected before $s$. 
\end{itemize}
\end{theorem}

\subsection{Extensions and open questions}\label{ss.extopen}

According to Lemma \ref{lem:comb_kappa}, the first bottleneck appears in the place where the graph $G$ has a bridge. It is a simple corollary that if the graph is 2-edge-connected, then it contains no temporal bottlenecks for the infection, which is specifically true for a cycle. 

It is possible to extend the result of Theorem~\ref{thm:main_smoothing} to the case of $\aa\in (1/d,1)$, when $d>2$, by defining for each finite connected graph $G$ with a root $s$ the number 
$$\kappa_d(G,s)=\min\limits_{e_1,\dots,e_{d-1}\in E(G)}|\CC(s,G\backslash \{e_1,\dots,e_{d-1}\})|.$$ 
Then, for every $k\leq \kappa_d(G,s)$, the expected time to infect $k$ vertices satisfies $\E T_k \leq C k^{1/\alpha}$. In order to obtain this result, one would first extend Lemma~\ref{lem.expectxy-t} to
$$\EB{\min\{X,Y_1-t,\dots,Y_{d-1}-t\} \md Y_1,\dots,Y_{d-1}>t}\asymp t^{1-\aa},$$
which has the same order of magnitude as before. Then, in Section~\ref{ss.kappa_result}, one should use the delayed process with always having $d$ active edges in the front, and the $Q$ process with $d-1$ always old edges, getting the bound $\E T_k \leq C k^{1/\alpha}$ at the end. However, it is less clear how to extend our speed-up results. In this case, the bottleneck-free subgraphs are $d$-edge-connected, hence adding an extra edge and creating a cycle may not always have a speed up effect, and more sophisticated constructions would need to be introduced. 

We have explained above why it is natural to consider {\it critical} random trees. Nevertheless, one could also condition a {\it subcritical} GW tree to be large, and then, depending on the offspring distribution, one would get a tree either very similar to a long path (with tiny trees hanging from it), or to a large star-graph (many spikes, all of bounded depth) \cite{Janson2012}, so in fact our two introductory examples basically cover subcritical trees, already. For {\it supercritical} GW trees, our questions are already not interesting: on the one hand, adding a single edge to an infinite tree has certainly an insignificant effect; on the other hand, as discussed above, for supercritical graphs there is a quite universal and fast spreading behaviour  \cite{bhamidi2014universality}, where a large tail for $\xi$ does not produce new phenomena.

About the models we have considered, one could also ask more detailed questions. Is there a {\it distributional scaling limit} for the process on a critical Galton-Watson tree conditioned to have depth at least $N$? Time should be scaled by $N^{1/\alpha}$, as for an iid sum of these variables (FPP on a half-line). And what is the scaling limit on the line segment $[-N/2,N/2]$, started from the origin, time scaled by $N^{1/\alpha}$?

In the proof of Theorem~\ref{thm:ER}, part~(1), we are basically showing that the largest cluster of $G\big(n,\frac{1}{n}+\frac{\lambda}{n^{4/3}}\big)$, viewed from a random vertex, resembles a critical GW tree conditioned to have large depth, and then we can use part~(1) of Theorem~\ref{thm:GW_smoothing}. With more work, one should be able to show that the {\it Benjamini-Schramm limit of the largest cluster} is in fact a GW tree with Poisson$(1)$ offspring, conditioned to be infinite. Surprisingly, we have not found a proof of this in the literature. Beyond the upper end of the critical window, the results of \cite{AnatomyYoung} do imply this claim.

\subsection{Notation}

Given two random variables $X_1$ and $X_2$, taking values in some spaces $\mathcal{X}_1$ and $\mathcal{X}_2$, the random variable $(\hat{X_1},\hat{X_2})$ with values in $\mathcal{X}_1\times\mathcal{X}_2$ is a \textit{coupling} of them if $\hat{X_i}$ has the same marginal distribution as $X_i$ for $i=1,2$, i.e., if $\Ps{\hat{X_i}\in E}=\Ps{X_i \in E}$ for all measurable subsets $E$ of $\mathcal{X}_i$.

If $X$ and $Y$ are two discrete random variables on the same space $\mathcal{X}$,  with distributions
$$\Ps{X=x}=p_x,\quad \Ps{Y=y}=q_y,\quad x,y\in\mathcal{X},$$
then the \textit{total variation distance} between the measures $p$ and $q$ is
$$d_{TV}(p,q)=\frac{1}{2}\sum_{x}|p_x-q_x|=\inf \Big\{\Ps{\hat{X}\neq\hat{Y}} : (\hat{X},\hat{Y})\text{ is a coupling of } X\text{ and }Y\Big\}\,.$$
The second equality is usually referred to as Strassen's theorem (see e.g. \cite{Strassen1965} or \cite{RemcoVDH}, p.59).

If $X$ and $Y$ are two real-valued random variables, we say that $Y$ \textit{stochastically dominates} $X$, denoted as $X\preceq Y$, if for every $x\in \R$, the following inequality holds:
$$\Ps{X\geqslant x}\leqslant \Ps{Y\geqslant x}.$$
Stochastic domination $X\preceq Y$ implies $\Es{X}\leqslant \Es{Y}$. We use the fact \cite{Lindvall1999} that $X$ stochastically dominates $Y$ if and only if there is a coupling $(\hat{X},\hat{Y})$ of $X$ and $Y$ such that $\Ps{\hat{X}\geqslant\hat{Y}}=1$.

If $f$ and $g$ are real-valued functions defined on $\N$ or $\R$, then $f\asymp g$ and $f=\Theta(g)$ mean that there exist constants $0<c<C<\infty$ such that $c f(x) < g(x) < C f(x)$ for all $x$. We write $f\sim g$, as $x\to x_0$, to mean that $\lim_{x\to x_0} f(x)/g(x)=1$, and write $f=o(g)$ to mean  that $\lim_{x\to x_0} f(x)/g(x)=0$.


\section{Two basic examples}\label{s.two}


\subsection{Spreading on a cycle}\label{ss.spreading_Z}

Consider the example of SI spreading with power law weights on the cycle $C_n$ with $n$ vertices, with $\aa\in(1/2,1)$. The process on the cycle can be understood via the process on the integer line $\Z$: the time $T_k$ of the infection of the $k$'th vertex has the same distribution on $\Z$ as on $C_n$, for every $k\leqslant n$. On $\Z$, we will denote by $X_i$ the power law distributed random weight on the edge $(i-1,i)$ if $i>0$, and on the edge $(i,i+1)$  if $i<0$. 

\begin{proposition}\label{prop:Z}
For the SI spreading process $(T_k)_{k=1}^{\infty}$ on $\Z$ with power law weights $\aa\in(1/2,1)$, the expected time to infect $k$ vertices satisfies
$$\Es{T_k}\asymp k^{1/\aa},$$
where the constant factors depend only on $\aa$.
\end{proposition}

\begin{proof}
Let $S_k = \sum\limits_{i=1}^{k}X_i$ and  $S^*_k = \sum\limits_{i=-1}^{-k}X_i$. Note that
$$ \min\{S_{\lfloor k/2 \rfloor}, S^*_{\lfloor  k/2 \rfloor }\}\leqslant T_{k+1} \leqslant \min\{S_k, S^*_k\}.$$ 
Then it is enough to prove that $\Es{\min\{S_k, S^*_k\}}\asymp k^{1/\aa}$. It is well-known (see \cite[Theorem 3.7.2]{Durrett2010}) that the sum $S_k$ as $k\to\infty$ is in the domain of attraction of the stable law $Y$ with the same parameter $\aa$:
$$\Ps{S_k/k^{1/\aa}>t}\xrightarrow[k\to\infty]{} \Ps{Y>t}.$$
Denote $\overline{S}_k=S_k/k^{1/\aa}$. The convergence is given via the convergence of characteristic functions, where the limit characteristic function is given by \cite{Durrett2010}: 
\begin{equation}\label{eq:char_fun}
\phi_Y(t)=\lim_{k\to\infty}\phi_{\overline{S}_k}(t)=C_{\sign(t)} \exp(-b |t|^{\aa}),
\end{equation}
where the constants, $C_{-1}=\overline{C}_{1}$ and $b>0$, depend on $\aa$. Hence, in the bounded interval $|t|<1$ the convergence in \eqref{eq:char_fun} is uniform in $t$, thus we can write
$$\phi_{\overline{S}_k}(t)=C_{\sign(t)} \exp(-b |t|^{\aa})(1+o(1)),$$
where $o(1)\to 0$ as $k\to\infty$ uniformly in $|t|<1$. Using the relation between the tail distribution and the characteristic function, given by the following inequality \cite[Eq. (3.3.1)]{Durrett2010}:
$$\Ps{|X|>2/u}\leqslant \frac{1}{u}\int\limits_{-u}^{u} \left(1-\phi_{X}(t) \right) dt,$$
where $X$ is a random variable with characteristic function $\phi_{X}(t)$, we derive that when $t$ is sufficiently large then for all $k$,
\begin{equation}
\begin{split}
\Ps{\overline{S}_k>t}\leqslant t \int\limits_{-2/t}^{2/t} 1- (1+o(1)) \, C_{\sign(t)} \exp\big(-b|x|^{\aa}\big) \, dx < t\int\limits_{-2/t}^{2/t}C_2|x|^{\aa} dx = C_3 t^{-\aa},
\end{split}
\end{equation}
where $C_3>0$ is constant that depends on $\aa$. Thus we have for sufficiently large $t$:
$$\Ps{\min\{S_k, S^*_k\}/k^{1/\aa}>t}\leqslant C_4 t^{-2\aa},$$
where $C_4>0$ is constant that depends on $\aa$. Since $S_k$ is positive then we can find a random variable $Z$ with power law tail with exponent $2\aa$ such that $|\min\{S_k, S^*_k\}/k^{1/\aa}|<Z$ a.s. for all $k>0$, and thus by Dominated Convergence Theorem for $\aa>1/2$ we have convergence of expectations  
$$\Es{\min\{S_k, S^*_k\}/k^{1/\aa}}\xrightarrow[k\to\infty]{} \Es{\min\{Y,Y^*\}}.$$
where $Y,Y^*$ are stable with parameter $\aa$. The minimum of $Y,Y^*$ has power law tail with exponent $2\aa$ thus has finite expectation and we have:
$$\Es{\min\{S_k, S^*_k\}/k^{1/\aa}}\asymp 1,$$
for all $k>0$, which implies the statement of the proposition.
\end{proof}


\subsection{Spreading on a star}\label{ss.star}

The star graph $\ST_n$ consists of a distinguished root vertex $0$ and vertices $\{1,2,\dots, n-1\}$ attached to it. On $\ST_n$, we consider the SI spreading process $T=(T_k)_{k=1}^n$ started from the root, with power law distributed random weights with $\aa\in(1/2,1)$, denoted as $X_1, X_2, \dots, X_{n-1}$. 

\begin{proposition}\label{prop:star}
On the graph  $\ST_n$ with $n\geqslant 2$, the expected time to infect $k$ vertices, for $k\leqslant n-1$, is bounded by
$$\Es{T_k}\leqslant C_\alpha k^{1/\aa},$$
where $C_\alpha>0$ is a constant that depends only on $\aa$. For $k=n-1$, in order to infect all but one vertices, we have
$$\Es{T_{n-1}}\asymp n^{1/\aa},$$
with the implicit constant factors depending on $\alpha$.
\end{proposition}

\begin{proof}
Denote by $X^{n-1}_{(1)}<\dots<X^{n-1}_{(n-1)}$ the order statistics of $X_1,\dots, X_{n-1}$. Then we have $T_{k+1}=X^{n-1}_{(k)}$, and it is obvious that
$$X^{n-1}_{(k)}\preceq X^{k+1}_{(k)},$$
for $k\leqslant n-2$, hence the second statement of the theorem implies the first one.

It is straightforward to calculate the tail distribution of $X^{k+1}_{(k)}$:
\begin{equation}\label{eq:star1}
\begin{aligned}
\Pb{X^{k+1}_{(k)}>t} &= 1-\Ps{X^{k+1}_{(k)}<t}\\
& = 1-(k+1)\,\Ps{X_1,\dots,X_k<t, X_{k+1}>t} - \Ps{X_1,\dots,X_{k+1}<t}\\
& = 1-(k+1)(1-t^{-\aa})^{k}t^{-\aa} - (1-t^{-\aa})^{k+1}. 
\end{aligned}
\end{equation} 

To get an upper bound on $\Es{X^{k+1}_{(k)}}$, we integrate~\eqref{eq:star1} over $t>0$, using the bound $(1-t^{-\aa})^k > 1- kt^{-\aa}$  for $t>k^{1/\aa}$, and the bound $\Ps{X^{k+1}_{(k)}>t}\leqslant 1$ for $t<k^{1/\aa}$. After cancellations,
\begin{equation*}
\begin{split}
\Eb{X^{k+1}_{(k)}} &\leqslant \int\limits_{0}^{k^{1/\aa}} 1\, dt + (k+1)k\int\limits_{k^{1/\aa}}^{\infty}t^{-2\aa}dt \\
& = k^{1/\aa} + (k+1)k\, \frac{k^{1/\aa-2}}{2\aa-1} \leqslant C_\alpha k^{1/\aa}.
\end{split}
\end{equation*}

In order to get a lower bound, we use the bound $(1-t^{-\aa})^k < 1- kt^{-\aa}+\frac{1}{2}k^2 t^{-2\aa}$ for $t>k^{1/\aa}$, and the bound $\Ps{X^{k+1}_{(k)}>t}\geq 0$ for $t<k^{1/\aa}$:
\begin{equation*}
\begin{split}
\Eb{X^{k+1}_{(k)}} &\geqslant \left( k(k+1)-\frac{(k+1)^2}{2} \right) \int\limits_{k^{1/\aa}}^{\infty}t^{-2\aa}dt-\frac{(k+1)k^2}{2} \int\limits_{k^{1/\aa}}^{\infty}t^{-3\aa}dt \\
& \sim\frac12\left( \frac{1}{2\aa-1}  - \frac{1}{3\aa-1} \right) k^{1/\aa}\geqslant c_\alpha k^{1/\aa},
\end{split}
\end{equation*}
with some $c_\alpha>0$. Thus, $$\Eb{X^{k+1}_{(k)}}\asymp k^{1/\aa},$$ which finishes proof of the theorem.
\end{proof} 


\section{General deterministic graphs}\label{s.determ}

\subsection{Preliminary lemmas}\label{ss.detprelim}

We present here two lemmas that will be used in tandem in the proof of Lemma~\ref{thm.expBound1-old&1-new}.

\begin{lemma}\label{lem.boundB_k}
Let $(b_n)_{n=1}^{\infty}$ be a positive sequence that satisfies the following recursive inequality for some $C>0$ and $0<\aa<1$:
\begin{equation}\label{eq.onb_k}
b_{n+1}\leqslant b_n + C b_n^{1-\aa}.
\end{equation}
Then
$$b_n\leqslant dn^{1/\aa},$$
with 
$d=\max\{b_1,(\aa C)^{1/\aa}\}$.
\end{lemma}

\begin{proof}
We prove the statement by induction. By definition, the statement holds for $b_1\leqslant d$. Suppose the statement holds for some $n>1$: for any $k$ with $1\leqslant k\leqslant n$, we have $b_k\leqslant dk^{1/\aa}$. Now rewrite~\eqref{eq.onb_k} as
$$b_{k+1}-b_k \leqslant C b_k^{1-\aa}.$$
Making a telescopic sum, then using the induction hypothesis and bounding the sum with an integral, we obtain 
\begin{equation}
\begin{split}
b_{n+1}-b_1 &\leqslant \sum\limits_{k=1}^{n} C b_k^{1-\aa} \leqslant \sum\limits_{k=1}^{n} C d^{1-\aa} k^{1/\aa-1}\leqslant \int\limits_{1}^{n+1} C d^{1-\aa}x^{1/\aa-1}dx\\
&=\aa C d^{1-\aa}\left((n+1)^{1/\aa}-1\right) \leqslant d\left((n+1)^{1/\aa}-1\right).
\end{split}
\end{equation}
Adding $b_1\leqslant d$ to this inequality, we arrive at $b_{n+1}\leqslant d(n+1)^{1/\aa}$, as desired.
\end{proof} 

\begin{lemma}\label{lem.expectxy-t}
Let $X$ and $Y$ be i.i.d. power law distributed random variables with $\aa\in (1/2,1)$. Then, for any $t>1$:
$$\Es{\min\{X,Y-t\}|Y>t}\asymp t^{1-\aa},$$
with the constant factors depending on $\alpha$.
\end{lemma}
\begin{proof}
The conditional tail distribution of the minimum is the following:
\begin{align*}
\Pb{\min\{X,Y-t\}>s \md Y>t} &= \frac{\Ps{X>s,Y>t+s}}{\Ps{Y>t}} = 
\begin{cases}
	\displaystyle s^{-\aa}\left(1+\frac{s}{t}\right)^{-\aa}, & \mbox{if } s>1;\\
	\displaystyle \left(1+\frac{s}{t}\right)^{-\aa}, & \mbox{if } 0<s<1.
\end{cases}
\end{align*}
Using the substitution $\displaystyle u=\frac{s}{t}$ we write the expected value as follows:
\begin{align*}
\Es{\min\{X,Y-t\}|Y>t} &= \int\limits_{0}^{\infty} \Pb{\min\{X,Y-t\}>s|Y>t} ds \\
 &= \int\limits_{0}^{1} \left(1+\frac{s}{t}\right)^{-\aa} ds + t^{-\aa} \int\limits_{1}^{\infty} \left(\frac{s}{t}\left(1+\frac{s}{t}\right)\right)^{-\aa}ds \\
 &= t\int\limits_{0}^{1/t} (1+u)^{-\aa}du + t^{1-\aa}\int\limits_{1/t}^{\infty} \left(u\left(1+u\right)\right)^{-\aa} du.
\end{align*}

In the first integral, $1\leqslant 1+u \leqslant 2$, hence that integral is $\asymp 1/t$ and the term is altogether $\asymp 1$. To calculate the second term, we split the interval of integration into two parts again:
\begin{align}\label{eq:techlem1}
t^{1-\aa}\int\limits_{1/t}^{\infty} \left(u\left(1+u\right)\right)^{-\aa}du = t^{1-\aa}\left[\int\limits_{1/t}^{1} \left(u\left(1+u\right)\right)^{-\aa}du + \int\limits_{1}^{\infty} \left(u\left(1+u\right)\right)^{-\aa}du\right].
\end{align}
The first integral on the RHS of \eqref{eq:techlem1} is less than $\int_0^1 u^{-\alpha} du = \frac{1}{1-\aa}$. 
The second integral can be bounded in the following way:
$$\int\limits_{1}^{\infty}(1+u)^{-2\aa}du\leqslant\int\limits_{1}^{\infty} \left(u\left(1+u\right)\right)^{-\aa}du\leqslant\int\limits_{1}^{\infty} u^{-2\aa}du,$$
which is clearly $\asymp \frac{1}{2\aa-1}$.
Summing up the three terms we have calculated, we have  proved the lemma. 
\end{proof}

\subsection{Model and main result}\label{ss.kappa_result}

We consider finite simple connected rooted graphs $G=(V,E,s)$ on $n$ vertices with $m$ edges. We denote by $V(G)$ the vertex set of $G$ and by $|G|$ the total size of it, or $|G|=|V|$. We call a path between $s$ and $t$ an \textit{$(s,t)$-path}. The \textit{graph distance} $d(s,t)$ between vertices $s$ and $t$ as the length of the shortest $(s,t)$-path.

Denote the weight, or the infection passage time, attached to the edge $e\in E$ as $X_e$, where $X_e$ is defined on the probability space $(\R_+,\F,\P)$. The probability space $\Omega=(\R_+^{m},\F^{m},\P^{m})$ is the space of all possible random assignments of weights to edges of the graph $G$ equipped with the product measure. Denote by $P(s,t)$ the shortest weighted $(s,t)$-path and by $|P(s,t)|$ the total weight of this path.  The $SI$ spreading process on the graph $G$ is a stochastic process $T=\{T_k:k\in\{1,\dots,n\}\}$, where $T_k$ is defined as the minimum over $H_k\in \GC_k$ of the maximum over vertices $t\in V(H_k)$ of the total weight of the shortest weighted $(s,t)$-path. In symbols,
$$T_k = \min_{H_k\subset \GC_k}\max_{t\in V(H_k)} |P(s,t)|,$$
where $\GC_k$ is the set of subtrees of $G$ on $k$ vertices with the same root $s$. We call an edge as \textit{occupied} at time $t$, if both end vertices are in the infected state $I$, \textit{unoccupied}, if both are in the susceptible state $S$, and \textit{active}, if one is in state $I$ and the other is in state $S$. 

The process $T$ is defined on the space $\Omega$, equipped with natural filtration $\F=\{\F_k:k\in\{1,\dots,n\}\}$. Denote the sample sequence $T(\omega)=\{T_k(\omega):k\in\{1,\dots,n\}\}$, where $\omega\in\Omega$. Each sample sequence $T(\omega),\, \omega\in\Omega$, defines an order $\epsilon_{T(\omega)}=(e^{\epsilon}_{1},e^{\epsilon}_{2},\dots,e^{\epsilon}_{m})$ on the edge set, in which the edges are occupied by the process. Given $T_k(\omega)$ we define for each $k$ the \textit{infection trail tree} $G(T_k(\omega))$ as the subtree from $\GC_k$ that consists of those edges that successfully passed the infection, i.e., the occupied edges on the shortest paths between the root $s$ and all $k-1$ other infected vertices. It may happen that at some time $T_k(\omega)$ two or more edges, incident to a newly infected vertex, become occupied at the same time. In this case we assume that first the edge on the shortest (weighted) path to the root is occupied, and the rest are occupied with respect to some fixed generic order $\overline{\epsilon}$ on the edge set to eliminate ambiguity. Also it may happen that some vertex has two or more shortest paths with equal total weight, but since the system is finite, this event happens with zero probability.

In the current framework, for each $\omega\in\Omega$ the spreading curve is the set of pairs $\{(\Es{T_k},k/n):1\leqslant k\leqslant n\}$. Our goal is to mathematically define the position of the first temporal bottleneck, responsible for the plateau on the spreading curve. Further we restrict ourselves to consideration of power law distributed weights with $\aa\in(1/2,1)$, however some lemmas consider broad class of weights.

\medskip
\noindent\textbf{First temporal bottleneck.} Remember we call an edge active at time $t$ if one of its incident vertices is in state $S$ and the other is in state $I$. In other words, an active edge is an edge that currently transmits an infection. Let $\omega\in\Omega$ and the \textit{front of the epidemic} $F(T_k(\omega))$ be the set of edges, that are active at time $T_k(\omega)$, where $k\in\{1,\dots,n\}$, in the sample sequence $T(\omega)$. Define $\kappa(G,s)$ to be the \textit{maximal} number of vertices $k$ such that for each sample sequence $T(\omega)$ and for each $j<k$, the front $F(T_j(\omega))$ has at least two active edges. In other words, it is the {\it minimal} $k$ such that there exists $\omega\in\Omega$ with $F(T_k(\omega))$ having one or zero active edges. We say the active edge $e$ is \textit{old} and has age $\tau>0$ at time $t$, if the edge has become active at time $t-\tau$. If $\tau=0$, then an edge is called \textit{new}. We now prove that if there is a sample sequence $T(\omega)$ and a number $i$, for which the front of the epidemics $F(T_i(\omega))$ has one active edge, then there is a large plateau at this point on the spreading curve.
\begin{lemma}\label{lem:SI_inf_expect}
Let $G$ be a finite rooted graph and let $T$ be the $SI$ spreading process with absolutely continuous passage times $\xi$, such that $\Es{\xi}=\infty$. Let there exists an $\omega_0\in\Omega$ and a number $i\in\N$, such that the front $F(T_i(\omega_0))$ has one active edge. Then for each $j$, where $i+1\leqslant j\leqslant n$, the expected infection time is $$\Es{T_j}=\infty.$$
\end{lemma}
\begin{proof}
The sample sequence $T(\omega_0)$ defines the order of occupation of the edge set $\epsilon_{T(\omega)}=(e^{\epsilon}_{1},e^{\epsilon}_{2},\dots,\\ e^{\epsilon}_{m})$. Since all edge weights have absolutely continuous distribution and the number of edges is finite, there exists a subset $\mathcal{A}(\omega_0)$ of sample sequences of positive measure with the same order of occupation of edges as in $\epsilon_{T(\omega)}$. More precisely, there exists a small $\varepsilon>0$ such that the set:
$$\mathcal{A}(\omega_0)=\{\omega: |X_{e^{\epsilon}_{j}}(\omega)-X_{e^{\epsilon}_{j}}(\omega_0)|<\varepsilon\},$$
which has positive measure for any $\varepsilon>0$. For this, one can take $\varepsilon$ to be smaller than half of the minimum of all the absolute differences between the edge weights of different edges (which is almost surely positive).

Then, since the front $F(T_i(\omega_0))$ has one active edge, then for each $\omega\in\mathcal{A}(\omega_0)$, the front $F(T_i(\omega))$ also has one active edge, say $e$. Therefore, we have
$$\Es{T_{i+1}-T_{i}\md \mathcal{A}(\omega_0)}=\Es{X_e}=\infty,$$
and by the law of total expectation
$$\Es{T_{i+1}}=\Es{T_{i+1}-T_{i}} + \Es{T_i} > \Es{T_{i+1}-T_{i}\md \mathcal{A}(\omega_0)}\Ps{\mathcal{A}(\omega_0)}=\infty.$$
Since $T_{i}\preceq T_{i+1}$ for all $i$, then we have $\Es{T_{j}}=\infty$, for all $j>i+1$, which finishes the proof of the lemma.
\end{proof} 
By Lemma~\ref{lem:SI_inf_expect}, $\kappa(G,s)$ captures the exact number of nodes that are infected before the possible emergence of the first plateau. It has a simple combinatorial description, as we stated in Lemma~\ref{lem:comb_kappa}.

\begin{proof}[Proof of Lemma~\ref{lem:comb_kappa}]
Suppose that there exists $\omega\in\Omega$ and the number $k$, where $0<k<n$, such that $F(T_{k}(\omega))$ has one active edge $e\in E$. Then at time $T_{k}(\omega)$ we can divide vertices of $G$ into two classes: infected (in state I) and susceptible (in state S). In the induced subgraph on infected vertices all edges are occupied, and in the subgraph on susceptible vertices all edges are unoccupied, and since the only edge is active, then it is the only one edge between these two subgraphs. Hence, the active edge $e$ is a cut edge, and the size of the infected subgraph equals to $k_{inf}=|\CC(s,G\backslash e)|$. Since $\kappa(G,s)$ is defined as the minimum over $\omega\in\Omega$, hence can be a bottleneck edge in the process, thus
$$\kappa(G,s)=\min_{e\in E(G)}|\CC(s,G\backslash e)|.$$
This finishes the proof of the Lemma.
\end{proof} 

\noindent\textbf{Delayed process $\T$.}  Fix an arbitrary order on the edge set $\epsilon=(e^{\epsilon}_{1},e^{\epsilon}_{2},\dots,e^{\epsilon}_{m})$. Define the process $\T = \{\T_k:k\in\{1,\dots,n\}\}$, coupled with the original process $T$ as follows. Start with the infected root $s$. Then choose the two active edges incident to $s$ with smallest indices in the order $\epsilon$ and spread the infection through them. At the time when one of these edges becomes occupied, choose the next active edge from $E$ with the smallest index in $\epsilon$ and repeat the procedure. If the two active edges share one susceptible vertex, then, when one edge gets occupied, choose two new active edges, given by the smallest indices in $\epsilon$. The process runs until there are no more new active edges to take, and the remaining times are assumed to be infinite. 

Obviously, for each $k<\kappa(G,s)$ and each $\omega\in\Omega$ the front $F(\T_k(\omega))$ has two active edges, since the delayed process can be turned into a particular realization of an original process with order of edge occupation $\epsilon$.

The process $\T$ stochastically dominates the process $T$, which is proved in the following lemma.
\begin{lemma}\label{lem:stochdom}
Let $G=(V,E,s)$ be a finite connected rooted graph and let $T$ be the SI spreading process on $G$ with absolutely continuous passage times $\xi$. Then the delayed process $\T$ stochastically dominates the process $T$.
\end{lemma}
\begin{proof}
Consider an $\omega\in\Omega$. Then for any $k$, where $1\leqslant k\leqslant n$, the sample sequence of the delayed process $\T_k(\omega)$ induces the infection trail tree $G(\T_k(\omega))$ and we have 
$$\T_k(\omega) = \max_{t\in V(G(\T_k(\omega)))} |P(s,t)|.$$
On the other hand, the original process is given by the minimum over all possible subtrees on $k$ vertices:
$$T_k(\omega) = \min_{H_k\in \GC_k}\max_{t\in V(H_k)} |P(s,t)|.$$
Therefore, we have $T_k(\omega)\leqslant\T_k(\omega)$ for all $k$, and, hence, $T\preceq\T$.
\end{proof} 

Despite the fact that the delayed process runs slower than the original, the next we define the process $Q$, which is even slower, but is necessary to achieve our final statement.

\medskip
\noindent\textbf{Process $Q$.} Define the $Q=\{Q_k:k\in\{1,\dots,n\}\}$ to be the stochastic process, in which at each time $Q_k$ there are two active edges with weights $X$ and $Y$ in the front: one of them is always old, with the age of the process, and an other is new. In symbols, let 
\begin{equation}\label{eq.procQ_k}
\begin{aligned}
Q_1 &= 0,\\
Q_2 &= \min\{X,Y\},\\
Q_{k+1} & = Q_k + \min\{X,Y-Q_k|Q_k,Y>Q_k\},
\end{aligned}
\end{equation}
where the unconditional $X,Y$ are i.i.d. edge weights. The process $Q$ qualitatively constitutes the particular scenario the infection can spread on $\Z$, having an ever old edge $Y$ and spreading only along new edges $X$ in one direction. The following lemma provides a bound on the expected time to infect $k$ vertices in the process $Q$.

\begin{lemma}\label{thm.expBound1-old&1-new}
Consider the process $Q_k$ defined above in \eqref{eq.procQ_k} with $X,Y\sim\pow(\aa)$. Then, for $\aa\in(1/2,1)$ and for each $k$, where $k\geqslant 1$, we have 
$$\Es{Q_k}\leqslant dk^{1/\aa}.$$
where $d>0$ is a constant that depends on $\aa$.
\end{lemma}

\begin{proof}
Using the law of total expectation, Lemma~\ref{lem.expectxy-t} and Jensen's inequality we have that
\begin{align*}
\Es{Q_{k+1}}&=\Es{Q_k+\min\{X,Y-Q_k\}} = \Es{Q_k} + \Eb{\Es{\min\{X,Y-Q_k\}|Q_k}} \\
&\leqslant \Es{Q_k} + C \Es{Q_k^{1-\aa}} \\
&\leqslant\Es{Q_k} + C \Es{Q_k}^{1-\aa},
\end{align*}
where $C>0$ is a large enough constant. Then immediately we have $\Es{Q_k}\leqslant b_k$, where $b_k$ is defined with a recursion
$$b_{k+1}=b_k+C b_k^{1-\aa},$$
$$b_1=\Es{Q_2}= \frac{2\aa}{2\aa-1}\leqslant (\aa C)^{1/\aa}:=d.$$
By Lemma~\ref{lem.boundB_k}, this sequence is bounded and we have for any $k\geqslant 1$:
$$\Es{Q_k}\leqslant d k^{1/\aa}.$$
This finishes the proof of the lemma.
\end{proof} 

Define the random variable $X_s$ with the following probability measure:
$$\Ps{X_s>t}:= \Ps{X-s>t|X>s},$$ 
where $s>0$. We call the random variable $X$ to have \textit{shifted power law} distribution $\shiftpow(\aa)$ with $\aa>0$, if $\Ps{X>t}=(t+1)^{-\aa}$, for all $t\geqslant 0$. Notionally, it represents a power law random variable with a shift of a size of the minimum cutoff. Note that for any $s>0$, if the random variable $X\sim \shiftpow(\aa)$ with $\aa>0$, then $X_s\stackrel{d}{=} (s+1)X$ and therefore, for any $s_1<s_2$:
\begin{equation}\label{eq:old_dom_new}
X_{s_1}\preceq X_{s_2}.
\end{equation}
In other words, if we consider the SI spreading with shifted power law passage times, then the older edges dominate the newer ones. We can now prove the main theorem of this section: 

\begin{proof}[Proof of Theorem~\ref{thm:main_smoothing}]
Let a random variable $X\sim\pow(\aa)$, then $(X-1)\sim \shiftpow(\aa)$. Define $T^{(X-1)}$ to be the SI spreading process  coupled to $T$, with shifted power law passage times with the same parameter $\aa$ as in $T$. The shift in times is deterministic, hence by time $T_k$, the process with shifted times is faster than the original process by the cumulative shift of not more than $k-1$:
$$T_k-T_k^{(X-1)}\leqslant k-1,$$
and equality holds only if the infection trail tree is a path of length $k-1$. The cumulative shift depends on the shape of the infection trail tree given by the process at time $T_k$ and can therefore be non-deterministic.

Since for $X,Y\sim \shiftpow(\aa)$, then by Lemma~\ref{lem:stochdom} the delayed process $\T^{(X-1)}$ with the same weights dominates the process $T^{(X-1)}$, and we have for any $k\leqslant\kappa(G,s)$:
$$T_k\preceq \T_k^{(X-1)}+(k-1).$$
Now consider the process $Q$ with shifted power law weights denoted as $Q^{(X-1)}$. Since in this case the old edges dominate newer ones, the1 process $Q^{(X-1)}$ dominates the process $\T^{(X-1)}$, hence we have
\begin{equation}\label{eq:shift_dom}
T_k\preceq Q_k^{(X-1)} +(k-1).
\end{equation}
Since the shift is negative, then we have $Q^{(X-1)}\preceq Q$. Hence, we have for each $k$, where $1\leqslant k\leqslant \kappa(G,s)$,
$$T_k\preceq Q_k+(k-1).$$
By Lemma \ref{thm.expBound1-old&1-new} we have
$$\Es{T_k}\leqslant d k^{1/\aa}+(k-1)\leqslant (d+1)k^{1/\aa},$$
where $d>0$ depends on $\aa$, which finishes the proof of the theorem.
\end{proof}
From the statement of Lemma~\ref{lem:comb_kappa} we have the following simple corollary.
\begin{corollary}
Let $G$ be a finite $2$-edge-connected rooted graph with root $s$ and let $T$ be the $SI$ spreading process on $G$ with power law weights, where $1/2<\aa<1$. Then for each $k\leqslant n$, 
$$\Es{T_k}\leqslant C k^{1/\aa}.$$
\end{corollary}

Thus, we established the existence of specific number of nodes $\kappa(G,s)$ for each finite connected rooted graph $G$, up to which a spreading curve of the SI process with power law passage times with $\aa\in (1/2,1)$ has no large plateaux. In case of trees, there is an analytic expression for $\kappa$. Consider a rooted tree $\TT$ with root $s$ on $n$ vertices. Let the degree of the root be equal to $d_s$. Denote subtrees hanging from the root as $\TT_1,\dots, \TT_{d_s}$ (see Figure~\ref{fig:cycle_trees_illustr}, a)). Then, by Lemma~\ref{lem:comb_kappa}, 
\begin{equation}\label{eq:kappa_tree}
\kappa(\TT,s)=|\TT|-\max\{|\TT_i|:1\leqslant i\leqslant d_s\}.
\end{equation}

\begin{figure}[h]
\centering
\begin{minipage}[b]{0.49\textwidth}
\centering
\includegraphics[width=0.68\linewidth]{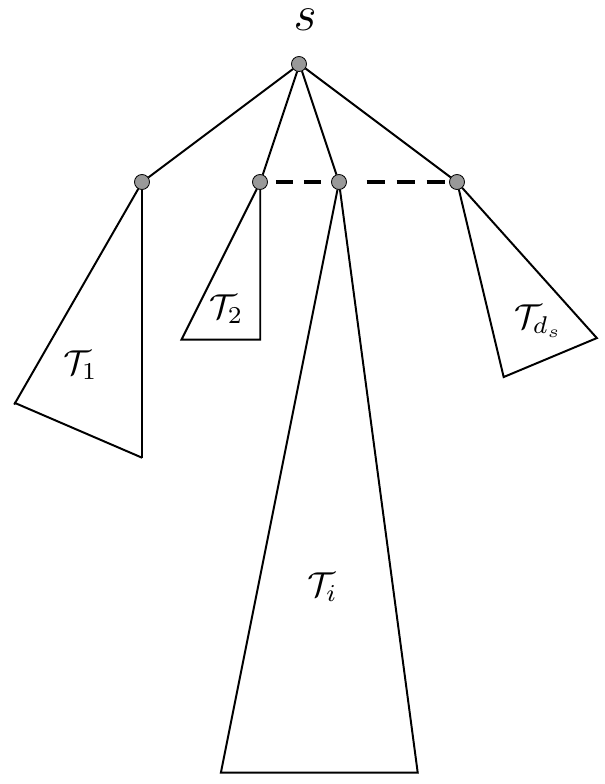}
\caption*{(a)}
\end{minipage}
\begin{minipage}[b]{0.49\textwidth}
\centering
\includegraphics[width=0.95\linewidth]{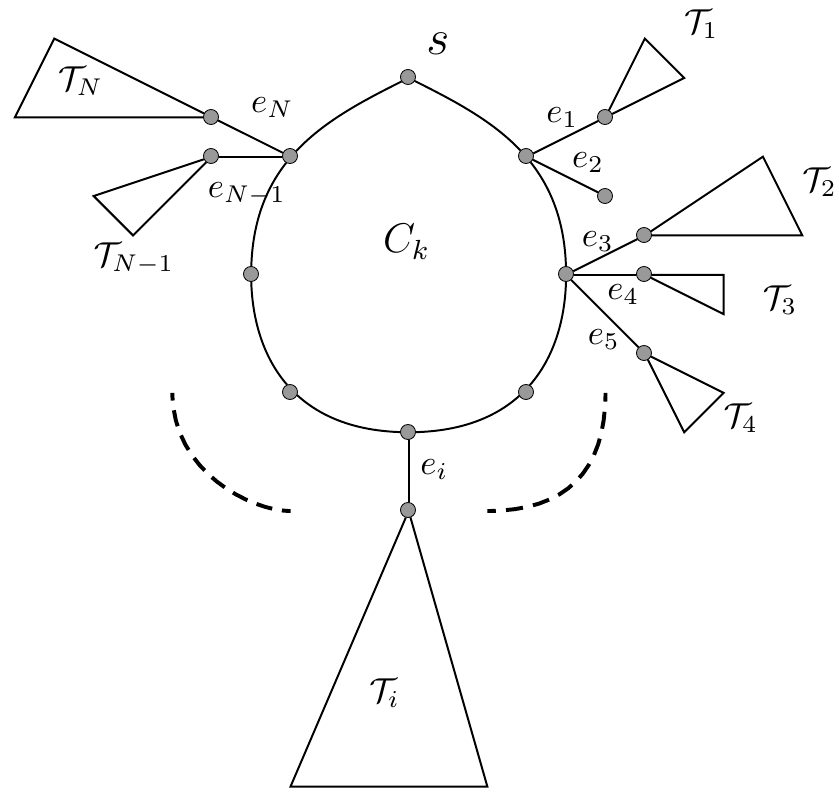}
\caption*{(b)}
\end{minipage}
\caption{Schematic structure: a) the tree $\TT$ with the root $s$ having degree $d_s$; b) the graph $\TT_{+e}$ of a tree $\TT$ with an extra edge $e$: a cycle $C_k$ with hanging trees $\TT_1,\TT_2,\dots,\TT_N$. The sizes of triangles represent the relative size of hanging trees.}
\label{fig:cycle_trees_illustr}
\end{figure}

Now add an extra edge $e$ between $s$ and a randomly chosen vertex in the tree. Then, we obtain a graph, denoted as $\TT_{+e}$, which consists of a cycle $C_k$ of length $k$ with $N$ rooted trees $\TT_1,\TT_2,\dots,\TT_N$, where $N\geqslant d_s$ a.s., attached to it by edges $e_1,e_2,\dots, e_N$ (see Figure~\ref{fig:cycle_trees_illustr}, b)). Then, by Lemma~\ref{lem:comb_kappa}, we have
\begin{equation}\label{eq:kappa_tree_edge}
\kappa(\TT_{+e},s)=|\TT_{+e}|-\max\{|\TT_i|:1\leqslant i\leqslant N\}.
\end{equation}
In the following sections we consider models of random trees, and we use the above formulas to compare how the addition of an extra edge changes $\kappa$.

\section{Critical Galton-Watson trees}\label{s.GW}

\subsection{Preliminaries}\label{ss.GWprelim}

A Galton-Watson (GW) tree $\TT$ is defined as a geneological tree of a \textit{Galton-Watson process} $\langle Z_n : n\geqslant 0\rangle$ of evolution of a particle system and we assume that $Z_0=1$ (see e.g. \cite{Janson2012} for the introduction). The process is characterized by the offspring distribution $\xi$, which is assumed in context of the paper to be non-degenerate, integer and have finite variance $\Var(\xi)<\infty$. The process (and the tree) is called \textit{critical}, if $\Es{\xi}=1$.

The GW trees can be viewed as rooted labeled trees. The root of the tree $\TT$ corresponds to the particle in the generation $Z_0$, and it is denoted by $<0>$. A generic particle in the generation $Z_{k}$ is indexed as $<0,l_1,\dots,l_k>$, where $l_r\geqslant 1$, $1\leqslant r\leqslant k$. The particles $<0,l_1,\dots,l_{k-1},j>$, where $j=1,2,\dots$, denote the children of the particle $<0,l_1,\dots,l_{k-1}>$ in generation $k-1$. Of course, not for all $j$ does $<0,l_1,\dots,l_{k-1},j>$ correspond to an actual vertex of $\TT$. Let $N(0,l_1,\dots,l_{k-1})$ be the number of children of $<0,l_1,\dots,l_{k-1}>$ in the process. Then $<0,l_1,\dots,l_{k-1},j>$ is a vertex of $\TT$ for $1\leqslant j\leqslant N(0,l_1,\dots,l_{k-1})$. 

Denote the set of all GW trees as $\langle GW \rangle$ and a randomly chosen GW tree as $\TT$. The size of $\TT$ is defined as the number of vertices it contains and is denoted as $|\TT|$. It is well known that a CGW tree is almost surely finite (e.g. Theorem 3.1, p. 84, \cite{RemcoVDH}) and the following theorem provides a bound on the size of a CGW tree \cite{Kolchin1986}.

\begin{lemma}\label{thm:size_trees}
Let $\TT$ be a CGW tree with integer offspring distribution $\xi$, such that $\Var(\xi):=\sigma^2<\infty$. Then for $n\to\infty$,
$$\Ps{|\TT|=n}=\frac{1}{\sqrt{2\pi}\sigma}n^{-3/2}(1+o(1)).$$
\end{lemma}

The \textit{height} $H(\TT)$ of a GW tree $\TT$ is the length of a longest path from the root or the maximum $N$, such that $Z_N>0$. The following limit theorem about the height of the tree $\TT$ holds \cite{Kesten1966}.

\begin{lemma}\label{thm:kolmogorov}
Let $\TT$ be a CGW tree with offspring distribution $\xi$, such that $\Var(\xi):=\sigma^2<\infty$. Then we have,
$$\lim_{N\to\infty}N\Ps{H(\TT)>N}=\lim_{N\to\infty}N\Ps{Z_N>0}=\frac{2}{\sigma^2}.$$
\end{lemma}

The following theorem provides an upper bound on the probability of having a tree of height at least $N$ conditioned on the exact size of this tree \cite{Janson2011}, \cite{Kolchin1986}.
\begin{lemma}\label{thm:width}
Let $\TT$ be a CGW tree with offspring distribution $\xi$, such that $\Var(\xi):=\sigma^2<\infty$. Then there exist positive constants $C$ and $c$, such that
$$\Ps{H(\TT)\geqslant x\big| |\TT|=n}\leqslant Ce^{-cx^2/n}.$$
\end{lemma}

We consider the set of CGW trees \textit{conditioned on $Z_N>0$}, where $N>0$, as the subset of trees $\langle GW \rangle$ with height at least $N$. Denote this set of conditioned CGW trees as $\langle GW \big| Z_N>0\rangle$. The expected limit size of the $k$'th generation in such trees is given in the following Theorem \cite{MeirMoon1978}.
\begin{lemma}\label{thm:exp_Z_k}
Let $\TT$ be a CGW tree with offspring distribution $\xi$, such that $\Var(\xi):=\sigma^2<\infty$. Then we have:
$$\lim_{N\to\infty}\Es{Z_k|Z_N>0} = 1+k\sigma^2.$$
\end{lemma}

Kesten in \cite{Kesten1986} proved that CGW trees, conditioned on $Z_N>0$, converge in distribution to an infinite CGW tree $\TT^{\infty}$, which is the geneological tree of the critical Galton-Watson process conditioned on \textit{non-extinction}. The infinite tree has the following construction. The tree $\TT^{\infty}$ has two types of vertices: \textit{normal} and \textit{special}, with root being special. Normal vertices have offsprings according to independent copies of $\xi$, while special nodes have a number of offsprings according to the size-biased distribution $\hat{\xi}$, where
$$\Ps{\hat{\xi}=k}:=kp_k,$$
and $k=0,1,2,\dots$. Every offspring of a normal vertex is normal. When a special vertex produces a number of offsprings, one of its children is selected uniformly at random and becomes special, while all other children are normal.

An alternative construction of the tree $\TT^{\infty}$ is to start by taking an infinite path $\gamma$ of special vertices from the root, which is called a \textit{spine}, and then attach $\nu=\hat{\xi}-1$ independent CGW trees at each node of the spine. Since each CGW tree is a.s. finite, it follows that $\TT^{\infty}$ a.s. has exactly one infinite path from the root, viz. the spine.

\subsection{Speed of convergence to the infinite tree}

We write as $(\TT[k]=T)$ and $(\TT^{\infty}[k]=T)$ the event that the first $k$ generations of the tree $\TT$ and $\TT^{\infty}$ respectively match the first $k$ generations of a given tree $T$. Denote by $\#T_k$ the size of $k$'th generation in the tree $T$. The following lemma holds for the trees $\TT$ and $\TT^{\infty}$~\cite{Kesten1986}.
\begin{lemma}\label{lem:measure}
Let $\TT$ be a CGW tree with offspring distribution $\xi$. Then, for any rooted vertex-labeled tree $T$ of at least $k$ generations:
$$\lim_{N\to\infty}\Ps{\TT[k]=T|Z_N>0} = \#T_k\cdot \Ps{\TT[k]=T}.$$
Then
$$\Ps{\TT^{\infty}[k]=T}=\#T_k\cdot \Ps{\TT[k]=T}.$$
\end{lemma}

It is natural that as $N\to\infty$ the conditioned tree $\TT^N:=(\TT|Z_N>0)$ and $\TT^{\infty}$ w.h.p.~start looking similar. The question now is how large (as a function of $N$) that similar part is. The following useful proposition answers this question. Versions of this proposition, with the tree conditioned to  have exactly $N$ vertices, have appeared in \cite[Theorem 5]{Kersting} and \cite[Theorem 6.5]{StuflerEnriched}, but we have not found our exact statement in the literature, hence include a proof.

\begin{proposition}\label{pr:coupling_T_T*}
Let $\TT^N$ be a CGW tree conditioned on $Z_N>0$ and $\TT^{\infty}$ be an infinite CGW tree. Then, as $N\to\infty$, for any $\varepsilon>0$ there exist $\delta>0$ and a coupling between $\TT^N$ and $\TT^{\infty}$, such that
$$\Ps{\TT^N[\delta N]\neq \TT^{\infty}[\delta N]}<\varepsilon.$$
\end{proposition}

\begin{proof}
In order to prove the statement of the lemma we show that the conditioned measure and the infinite measure are close in total variation distance. First we establish bounds on the conditioned measure. Consider a rooted tree $T$ with height $k$, where $k\leqslant\delta N$ and $\delta>0$ is small. Then by Bayes' formula,
\begin{equation}\label{eq:bayes}
\begin{split}
\Ps{\TT[k]=T|Z_N>0} &= \frac{\Ps{Z_N>0|\TT[k]=T}}{\Ps{Z_N>0}}\Ps{\TT[k]=T} \\
&= \frac{\Ps{Z^{(1)}_{N-k}>0\cup \dots\cup Z^{(\#T_k)}_{N-k}>0}}{\Ps{Z_N>0}}\Ps{\TT[k]=T},
\end{split}
\end{equation}
where $Z^{(i)}_{N-k}$ denotes the $(N-k)$'th generation in the copy of the CGW process $Z^{(i)}$, started from a vertex at level $k$. By Lemma~\ref{thm:kolmogorov}, for a large enough $N$ there exists $\varepsilon_0>0$ such that,
\begin{equation}\label{eq:kolmog_lim_est}
\frac{2}{\sigma^2 N}(1-\varepsilon_0)<\Ps{Z_N>0}<\frac{2}{\sigma^2 N}(1+\varepsilon_0).
\end{equation}
Also, when $N-k$ is large enough, there exists $\varepsilon_1>0$, such that
\begin{equation}\label{eq:kolmog_lim_est2}
\frac{2}{\sigma^2 (N-k)}(1-\varepsilon_1)<\Ps{Z^{(i)}_{N-k}>0}<\frac{2}{\sigma^2 (N-k)}(1+\varepsilon_1),
\end{equation}
where $1\leqslant i\leqslant \#T_k$. In order to simplify the further calculations we take common $\varepsilon_2:=\max(\varepsilon_0,\varepsilon_1)$ instead of $\varepsilon_0$ and $\varepsilon_1$ in \eqref{eq:kolmog_lim_est} and \eqref{eq:kolmog_lim_est2}, and the conditioned measure is denoted as $\Pns{\cdot}:=\Ps{\cdot\md Z_N>0}$. 

\medskip
\noindent\textbf{Upper bound.} We use the union bound on the right-hand side of \eqref{eq:bayes} and together with \eqref{eq:kolmog_lim_est} and \eqref{eq:kolmog_lim_est2}, we obtain
\begin{equation}\label{eq:estim1}
\begin{split}
\Pns{\TT[k]=T} &\leqslant \frac{\#T_k\Ps{Z_{N-k}>0}}{\Ps{Z_N>0}}\Ps{\TT[k]=T} \\
&< \frac{N}{(N-k)}\#T_k\Ps{\TT[k]=T}\frac{1+\varepsilon_2}{1-\varepsilon_2}.
\end{split}
\end{equation}
Therefore, we can write that for small enough $k$ there exists $\varepsilon_3>0$, such that
\begin{equation}\label{eq:estim2}
\Pns{\TT[k]=T} < \frac{N}{(N-k)}\#T_k\Ps{\TT[k]=T}(1+\varepsilon_3).
\end{equation}

\noindent\textbf{Lower bound.} We rewrite \eqref{eq:bayes} using \eqref{eq:estim2} as follows:
\begin{equation}\label{eq:bayes1}
\begin{split}
\Pns{\TT[k]=T} &= \frac{1}{\Ps{Z_N>0}}\left(1-\Ps{Z^{(1)}_{N-k}=0\cap \dots\cap Z^{(\#T_k)}_{N-k}=0}\right)\Ps{\TT[k]=T} \\
&=\frac{1}{\Ps{Z_N>0}}\left(1-\left(1-\Ps{Z_{N-k}>0}\right)^{\#T_k}\right)\Ps{\TT[k]=T} \\
&>\frac{1}{\Ps{Z_N>0}}\left(1-\left(1-\frac{2(1-\varepsilon_1)}{\sigma^2 (N-k)}\right)^{\#T_k}\right)\Ps{\TT[k]=T}.
\end{split}
\end{equation}
Since for any $x$, where $x>0$, we have $1-x< \exp(-x)<1-x+x^2/2$, then for any $n\geqslant 1$
\begin{equation}\label{eq:exponent2}
1-(1-x)^n>1-\exp(-nx)>nx-\frac{(nx)^2}{2}.
\end{equation}
We rewrite \eqref{eq:bayes1} using \eqref{eq:exponent2} for $x=\Ps{Z_{N-k}>0}$ and $n=\#T_k$ as follows:
\begin{equation}\label{eq:bayes2}
\begin{split}
\Pns{\TT[k]=T} &>\frac{\Ps{\TT[k]=T}}{\Ps{Z_N>0}}\left(\#T_k\Ps{Z_{N-k}>0}-\frac{1}{2}\left(\#T_k\Ps{Z_{N-k}>0}\right)^2\right).
\end{split}
\end{equation}
Now use \eqref{eq:kolmog_lim_est2} and we obtain:  
\begin{equation}\label{eq:bayes3}
\begin{split}
\Pns{\TT[k]=T} &> \frac{\Ps{\TT[k]=T}}{\Ps{Z_N>0}}\left(\frac{2\#T_k(1-\varepsilon_2)}{\sigma^2 (N-k)}-\frac{1}{2}\left(\frac{2\#T_k(1+\varepsilon_2)}{\sigma^2 (N-k)}\right)^2\right) \\
&>\Ps{\TT[k]=T}\#T_k\left(\frac{N}{N-k}\frac{1-\varepsilon_2}{1+\varepsilon_2}-\#T_k\frac{C_1N}{(N-k)^2}\right),
\end{split}
\end{equation}
where $C_1=\frac{(1+\varepsilon_{2})}{2\sigma^2}<1/\sigma^2$. Therefore, we can write that for small enough $k$ there exists $\varepsilon_4>0$ and a bounded $C_2>0$, that depends on $\sigma$ and $\varepsilon_2$, such that
\begin{equation}\label{eq:estim3}
\Pns{\TT[k]=T} > \Ps{\TT[k]=T}\#T_k\left(\frac{N}{N-k}-\#T_k\frac{C_2 N}{(N-k)^2}\right)(1-\varepsilon_4).
\end{equation}
Combining the \eqref{eq:estim2} and \eqref{eq:estim3}, and choosing $\varepsilon_5:=\max\{\varepsilon_3,\varepsilon_4\}$, we obtain the following bounds on the probability $\Pns{\TT[k]=T}$:
\begin{equation}\label{eq:bound_on_p_inf}
\left(\frac{N}{N-k}-\#T_k\frac{C_2N}{(N-k)^2}\right)(1-\varepsilon_5)\leqslant\frac{\Pns{\TT[k]=T}}{\#T_k\Ps{\TT[k]=T}}\leqslant\frac{N}{(N-k)}(1+\varepsilon_5).
\end{equation}

\noindent\textbf{Total variation distance.} Now we bound the total variation distance between conditioned and infinite measures. From the upper bound in \eqref{eq:bound_on_p_inf} we obtain that when $k$ is small enough, the following inequality holds:
\begin{equation*}
\begin{split}
\Pns{\TT[k]=T}- \Ps{\TT^{\infty}[k]=T}& \leqslant \left(\frac{N}{N-k}(1+\varepsilon_5)-1\right)\#T_k\Ps{\TT[k]=T} \\
&=\left(\left(\frac{N}{N-k}-1\right)+ \frac{N}{N-k}\varepsilon_5\right)\#T_k\Ps{\TT[k]=T},
\end{split}
\end{equation*}
and, on the other hand, from the lower bound in \eqref{eq:bound_on_p_inf} we obtain
\begin{equation*}
\begin{split}
\Ps{\TT^{\infty}[k]=T}- \Pns{\TT[k]=T}\leqslant & \left( 1-\frac{N}{N-k}(1-\varepsilon_5)+\#T_k\frac{C_{\varepsilon}N}{(N-k)^2}(1-\varepsilon_5)\right)\cdot \\ 
&\cdot \#T_k\Ps{\TT[k]=T} \\
=&\left(\left( 1-\frac{N}{N-k}\right)+\frac{N}{N-k}\varepsilon_5+\#T_k\frac{C_{\varepsilon}N}{(N-k)^2}(1-\varepsilon_5)\right)\cdot \\ 
&\cdot \#T_k\Ps{\TT[k]=T}.
\end{split}
\end{equation*}
Comparing both bounds we see that all summands are positive, except of $\left(1-\frac{N}{N-k}\right)$, thus we can inverse the sign it and derive the bound for an absolute value. Summing those bounds over all possible trees of height $k$, we obtain
\begin{equation}\label{eq:bound_measure1}
\begin{split}
\sum\limits_{T}\Big|\Pns{\TT[k]=T}- \Ps{\TT^{\infty}[k]=T}\Big|\leqslant & \sum\limits_{T}\left(\left( \frac{N}{N-k}-1\right)+\varepsilon_5\frac{N}{N-k}+\#T_k\frac{C_{\varepsilon}N}{(N-k)^2}(1-\varepsilon_5)\right)\cdot \\ 
&\cdot \#T_{k}\Ps{\TT[k]=T}.
\end{split}
\end{equation}
From the fact that we have a measure on the set of infinite trees and from Lemma~\ref{thm:exp_Z_k} we have 
$$\sum\limits_{T}\Ps{\TT^{\infty}[k]=T}=\sum\limits_{T}\#T_{k}\Ps{\TT[k]=T} = 1,$$ 
$$\Es{Z_k\md \TT^{\infty}}=\sum\limits_{T}(\#T_{k})^2 \Ps{\TT[k]=T} = 1+k\sigma^2.$$
Therefore we can rewrite \eqref{eq:bound_measure1} for $k=\delta N$, when $\delta>0$ is small, and obtain
\begin{equation*}
\begin{split}
\sum\limits_{T}\Big|\Pns{\TT[\delta N]=T}- \Ps{\TT^{\infty}[\delta N]=T}\Big|\leqslant & \frac{\delta}{1-\delta}+\frac{\varepsilon_5}{1-\delta}+C'_2\delta\frac{1-\varepsilon_5}{(1-\delta)^2}+\frac{C_2}{N}\frac{1-\varepsilon_5}{(1-\delta)^2},
\end{split}
\end{equation*}
where $C'_2=C_2\sigma^2<1$. Hence, for any $\varepsilon_6>0$ we can find large $N$ and small $\delta>0$, such that
\begin{equation}\label{eq:bound_abs}
\begin{split}
\sum\limits_{T}\Big|\Pns{\TT[\delta N]=T}- \Ps{\TT^{\infty}[\delta N]=T}\Big|& \leqslant \varepsilon_6.
\end{split}
\end{equation}
Denote the projection of measures $\Pn$ and $\Pinf$ onto the trees with common first $\delta N$ layers $\TT[\delta N]$ as $\Pn\restr{\delta N}$ and $\Pinf\restr{\delta N}$ respectively. Then, by \eqref{eq:bound_abs} and the definition of the total variation distance we have
\begin{equation*}
d_{TV}\left(\Pn\restr{\delta N},\Pinf\restr{\delta N}\right)\leqslant \frac{1}{2}\varepsilon_6.
\end{equation*}
Hence by Strassen's Theorem there exists a coupling of random variables $\TT[\delta N]$ and $\TT^{\infty}[\delta N]$, with the same $d_{TV}$. This finishes the proof of Proposition~\ref{pr:coupling_T_T*}.
\end{proof}

\subsection{Proof of Theorem~\ref{thm:GW_smoothing}}

\begin{proof} {\bf (1)} Consider first the infinite CGW tree $\TT^{\infty}$. Following the construction of $\TT^{\infty}$ we denote the spine as $\gamma$ and label the unconditioned CGW trees attached to the root as $t_1, t_2,\dots, t_{\nu}$, and the rest of the tree as $t_0$. The number of the unconditioned trees $\nu$ is represented by shifted size-biased i.i.d. random variables$\nu=\hat{\xi}-1$. Let $\mu:=\E\nu$, then it is straightforward to show that $\mu=\sigma^2<\infty$. By Markov's inequality for any given $C_1>0$,
$$\Ps{\nu>C_1} < \frac{\sigma^2}{C_1}.$$
Hence, for any $\varepsilon_1>0$ there exists $C_1>0$, such that
$$\Ps{\nu<C_1} > 1-\varepsilon_1.$$
Then, using the law of total probability we bound the total size of these trees in the following way:
\begin{equation}\label{eq:tree_size}
\begin{split}
\PB{\sum\limits_{i=1}^{\nu}|t_i|>C_1 K} &=\PB{\sum\limits_{i=1}^{\nu}|t_i|>C_1 K\md \nu<C_1}\Ps{\nu<C_1}\\
&\quad + \PB{\sum\limits_{i=1}^{\nu}|t_i|>C_1 K\md \nu>C_1}\Ps{\nu>C_1} \\
&< \PB{\sum\limits_{i=1}^{C_1}|t_i|>C_1 K} +\varepsilon_1.
\end{split}
\end{equation}
It remains to show that the total size of $C_1$ trees is bounded. Using the union bound we can write:
\begin{equation*}
\PB{\sum\limits_{i=1}^{C_1}|t_i|>C_1 K} < \Ps{\mbox{at least one of }|t_i|>K} < C_1\Ps{|t_i|>K},
\end{equation*}
and by Lemma~\ref{thm:size_trees} and bounding the sum with an integral, we obtain that for large $K$,
\begin{equation}\label{eq:gw_size_upper_bound}
\Ps{|t_1|>K}\leqslant C_2\sum_{k=K+1}k^{-3/2}
\leqslant C_2\int\limits_{K}^{\infty} x^{-3/2}dx= \frac{C_3}{\sqrt{K}},
\end{equation}
where $C_3>0$ is constant that depends on $\xi$. Hence, for any $\varepsilon_2>0$ and given $C_1>0$ there exists large $K$, such that 
\begin{equation}\label{eq:total_size_t_k}
\PB{\sum\limits_{i=1}^{\nu}|t_i|<C_1 K}> 1-\varepsilon_2.
\end{equation}
Since the total size of trees $t_i$, where $i\geqslant 1$, is bounded w.h.p., then their height is bounded too. By Proposition~\ref{pr:coupling_T_T*}, for large $N$ and any $\varepsilon_3>0$ we can find $\delta>0$ and a coupling of $\TT^N$ and $\TT^{\infty}$, such that the tree $\TT^N[\delta N]$ is same as $\TT^{\infty}[\delta N]$ with probability at least $(1-\varepsilon_3)$, hence by \eqref{eq:total_size_t_k} for any $\varepsilon_4>0$ and large $N$ there exists $K>0$ such that in the conditioned tree $\TT^{N}$:
\begin{equation}\label{eq:total_size_t_k_inN}
\PnB{\sum\limits_{i=1}^{\nu}|t_i|<C_1 K}> 1-\varepsilon_4.
\end{equation}

That is, the total size of trees $t_i$, where $1\leqslant i\leqslant \nu$, is absolutely bounded. On the other hand, as $N\to\infty$, the conditioning $Z_N>0$ ensures that the total size of $\TT^N$ grows to infinity, which then must be due to the tree $t_0$. Thus, by formula~\eqref{eq:kappa_tree}, where $\TT_i=t_i$ for $i\geqslant 0$, we have that $\kappa(\TT^N,s)$ is exactly equal to $\sum_{i=1}^{\nu}|t_i|$. Therefore,~\eqref{eq:total_size_t_k_inN} implies the tightness $\kappa(\TT^N,s)$, as claimed. 
\medskip

\noindent {\bf (2)} In order to prove the second statement, we make use of formula (\ref{eq:kappa_tree_edge}). We consider the fraction of the conditioned tree that w.h.p.~resembles the infinite tree up to the depth $\delta N$, where $\delta>0$, then show that the total size of the trees that hang off the first $\delta^2 N$ vertices on the spine constitutes a non-zero but small fraction of the total size of the conditioned tree w.h.p.  Thus, the probability that the extra edge is attached to any of these trees is small, and hence all these trees will hang off from the cycle created, ensuring that  $\kappa$ is positive. 

We will first consider the infinite tree $\TT^{\infty}$, then will translate our results using Proposition~\ref{pr:coupling_T_T*}. Denote as $\gamma_k$, where $k\geqslant 1$, the initial part of $\gamma$ between the root and the vertex at depth $k$ inclusive. Following the construction of $\TT^{\infty}$,  the number of unconditioned trees attached to $\gamma_k$ is $S_k=\sum_{i=1}^{k}\nu_i$, where $\nu_i=\hat{\xi_i}-1$, and $\hat{\xi_i}$ is the shifted size-biased version of the offspring distribution $\xi$. Label these unconditioned trees in the breadth-first order as $t_1,t_2,\dots,t_{\nu_1},\dots,t_{S_k}$. 

Let $n:=\delta N$ and $n':=\delta^2 N$. Consider the $S_{n'}$ unconditioned trees that hang off $\gamma_{n'}$, and let $t_0:=\TT^{\infty}\backslash \left(\gamma_{n'}\cup t_1\cup\dots\cup t_{S_{n'}}\right)$. We show that the $t_i$'s, where $1\leqslant i\leqslant S_{n'}$, do not go deeper than generation $n$  in $\TT^{\infty}$ w.h.p., or in other words, they have height at most $n-n'$. To start with, by Lemma \ref{thm:kolmogorov}, for any $\varepsilon'_1>0$, small $\delta>0$ and $i$, for all large enough $N$ we have 
$$\Ps{H(t_i)>n-n'}=\Ps{H(t_i)>(1-\delta)\delta N}<\frac{2}{\sigma^2(1-\delta)\delta N}(1+\varepsilon'_1).$$
We will now prove that the number of unconditioned trees $S_{n'}$ is linear in $n'$ w.h.p. Remember that $\mu=\Es{\nu}<\infty$. Hence, by the SLLN, for any $\varepsilon'_2>0$ and any $0<k'_1<1<K'_1<\infty$, if $n'$ is large enough, then 
\begin{equation}\label{eq:hangingSLLN}
\Ps{ k'_1 \mu n' < S_{n'} < K'_1\mu n'}>1-\varepsilon'_2.
\end{equation}
Hence, using the law of total probability and a union bound, we have 
\begin{align*}
\PB{\exists i\in \{1,\dots,\nu_{n'}\}:\,\,H(t_i)>n-n'} &< \PB{\exists i\in \{1,\dots,K'_1\mu n'\}:\,\,H(t_i)>n-n'}+ \varepsilon'_2\\
&<\frac{2K'_1\mu\delta }{\sigma^2(1-\delta)}(1+\varepsilon'_1) + \varepsilon'_2\,.
\end{align*}
Thus, for any $\varepsilon'_3>0$ there exists $\delta>0$ small enough, such that 
$$
\Pb{H(t_i)<n-n'\,\, \forall i\in \{1,\dots,\nu_{n'}\}}>1-\varepsilon'_3.
$$

We will now show that the total size of $S_{n'}$ unconditioned trees is w.h.p.~of order $n'^2$. Indeed, for any $0<k'_2<K'_2<\infty$, using~\eqref{eq:hangingSLLN} again,
\begin{equation}\label{eq:tree_size2}
\Ps{\sum\limits_{i=1}^{S_{n'}}|t_i|>K'_2 n'^2} < \Ps{\sum\limits_{i=1}^{K'_1\mu n'}|t_i|>K'_2 n'^2} +\varepsilon'_2,
\end{equation}
and similarly,
\begin{equation}\label{eq:tree_size3}
\Ps{\sum\limits_{i=1}^{S_{n'}}|t_i|<k'_2 n'^2} < \Ps{\sum\limits_{i=1}^{k'_1\mu n'}|t_i|<k'_2 n'^2} + \varepsilon'_2.
\end{equation}
Lemma~\ref{thm:size_trees} implies that $\sum |t_i|$ belongs to the domain of attraction of L\'evy's stable law with exponent $1/2$; see \cite[Theorem 3.7.2]{Durrett2010}. Therefore, $\sum_{i=1}^{k'n'} |t_i| / (k'n')^2$ has a non-degenerate limit law that is independent of $k'$. It follows that if $K_2'$ is chosen large enough in~(\ref{eq:tree_size2}), and $k_2'$ is chosen small enough in~(\ref{eq:tree_size3}), then the first terms on the right hand sides become small. Altogether,
\begin{equation}\label{eq:GW_eps4}
\PB{\frac{1}{n'^2}\sum\limits_{i=1}^{S_{n'}} |t_i|\in [k'_2,K'_2]} > 1-\varepsilon'_{4}.
\end{equation}

Now, by Proposition~\ref{pr:coupling_T_T*}, for large $N$ and any $\varepsilon'>0$ we can find $\delta>0$ and a coupling of $\TT^N$ and $\TT^{\infty}$, such that the tree $\TT^N[n]$ is same to $\TT^{\infty}[n]$ with probability at least $1-\varepsilon'$. There exists an image of $\gamma_{n}$ in $\TT^N$, and all trees that hang off $\gamma_{n'}$ do not go deeper than generation $Z_n$ w.h.p. This and~\eqref{eq:GW_eps4} imply that for any $\varepsilon'_5>0$ and large $N$ there exist $k'_2,K'_2>0$ such that in the conditioned tree $\TT^{N}$:
\begin{equation}\label{eq:GW_eps5}
\PnB{\frac{1}{n'^2}\sum\limits_{i=1}^{S_{n'}} |t_i|\in [k'_2,K'_2]} > 1-\varepsilon'_{5}.
\end{equation}

Now we are ready to prove that the extra edge is attached to the subtree $t_0$ in $\TT^N$ w.h.p. The probability of this event can be bounded as follows, for any $\delta_0>0$:
\begin{equation}\label{eq:GW_P_attach_edge}
\begin{split}
\Ps{e\in t_0} &>  \PB{e\in t_0\md \frac{|\TT^N_{+e}\backslash t_0|}{|\TT^N_{+e}|}<\delta_0} \, \PB{\frac{|\TT^N_{+e}\backslash t_0|}{|\TT^N_{+e}|}<\delta_0}\\
&> (1-\delta_0) \, \PB{\frac{|\TT^N_{+e}\backslash t_0|}{|\TT^N_{+e}|}<\delta_0}\,,
\end{split}
\end{equation}
where we used that the attachment is uniform.

Now notice that~\eqref{eq:GW_eps5} implies, in particular, that for any $\varepsilon_5>0$ there is a small enough $\beta>0$ such that 
\begin{equation}\label{eq:smallT}
\Pns{|\TT^N|>\beta N^2}> 1-\varepsilon_5.
\end{equation}
Using \eqref{eq:smallT} and \eqref{eq:GW_eps4}, we have
$$
\PB{\frac{|\TT^N_{+e}\backslash t_0|}{|\TT^N_{+e}|}<\frac{K'_2\delta^4}{\beta}}>1-\varepsilon_5-\varepsilon'_{5}.
$$
Since $\delta$ can be chosen to be small enough, combining this with~\eqref{eq:GW_P_attach_edge} we get that for any $\varepsilon'_{6}>0$ there exists $\delta'>0$  such that for all large enough $N$, 
\begin{equation}\label{eq:attach}
\Ps{e\in t_0}>1-\varepsilon'_{6}.
\end{equation}

Now since the statement of the theorem about the fraction of $\kappa(\TT^N_{+e},s)$ and $|\TT^N_{+e}|$, the large volume of the tree can decrease this fraction, thus we need to prove that the probability to have a very large tree $\TT^N$ on the $N^2$ scale is also small. Let us first write Bayes' formula:
\begin{equation*}
\Pns{|\TT^N|=M}=\Ps{|\TT|=M\md Z_{N}>0} = \frac{\Ps{Z_{N}>0\md |\TT|=M}}{\Ps{Z_N>0}}\Ps{|\TT|=M}\,.
\end{equation*}
Since the condition $Z_N>0$ implies that $|\TT^N| > N$, we may assume that $M > N$. Then, by Lemmas~\ref{thm:size_trees},~\ref{thm:kolmogorov} and~\ref{thm:width}, it follows that
\begin{equation}\label{eq:GW_size_T_bound}
\Pns{|\TT^N|=M} < C_1 e^{-c_1\frac{N^2}{M}}NM^{-3/2}.
\end{equation}
Taking a large $B>0$, we thus have:
\begin{equation*}
\begin{split}
\Ps{|\TT^N_{+e}|> B N^2} &< C_1\sum\limits_{m>B N^2}e^{-c_1\frac{N^2}{m}}Nm^{-3/2} \\
&< C_1\int\limits_{B N^2}^{\infty}e^{-c_1\frac{N^2}{m}}Nm^{-3/2}dm =C_1\int\limits_{B}^{\infty}e^{-\frac{c_1}{x}}x^{-3/2}dx < C_2 \, B^{-1/2}.
\end{split}
\end{equation*}
Hence, for any $\varepsilon_6>0$ we can find $B>0$ such that
\begin{equation}\label{eq:largeT}
\Ps{|\TT^N_{+e}|<B N^2}> 1-\varepsilon_{6}.
\end{equation}

Now use formula~\eqref{eq:kappa_tree_edge} for $\kappa(\TT^N_{+e},s)$. If the other endpoint of the extra edge $e$ is in $t_0$, then  $\kappa(\TT^N_{+e},s)> \min\big\{|t_0|,\sum_{i=1}^{S_n'}|t_i|\big\}$, and if $\delta$ is chosen small enough so that $\beta > 2K'_2\delta^4$, and the events of~\eqref{eq:smallT} and~\eqref{eq:GW_eps4} are satisfied,  then this minimum is $\sum_{i=1}^{S_n'}|t_i|$. That is, combining everything, we have 
$$
\PB{\frac{\kappa(\TT^N_{+e},s)}{|\TT^N_{+e}|}>\frac{k'_2\delta^4}{B}}>1-\varepsilon_{5}-\varepsilon'_{5}-\varepsilon_{6}-\varepsilon'_{6}.
$$
This finishes proof of the second part of Theorem~\ref{thm:GW_smoothing}.
\end{proof}

\section{Uniform spanning tree of the complete graph}\label{s.UST}

The uniform spanning tree $\UST(n)$ is a uniform random element of the set of spanning trees of the complete graph $K_n$ on $n$ vertices (of which there are $n^{n-2}$, as Cayley's formula tells us). Marking a vertex $s$ of $K_n$ as the root, we have a uniform rooted spanning tree. This random tree model can be considered as a variation of the first one: instead of producing a ``large critical tree'' by conditioning a critical Galton-Watson tree to have depth larger than a given value, one could condition on having a given volume. In particular, a critical Poisson Galton-Watson tree, conditioned to have volume $n$, turns out to be exactly the rooted $\UST(n)$; see, e.g., \cite{Janson2012} for this equivalence. 

In this section we derive the analogous result about speeding up of the SI spreading with power law weights on the rooted $\UST(n)$ with an extra edge added between the root and a random vertex in the tree. This model could probably be studied using variations of the methods of the previous section, but we use a rather different approach, based on Loop-Erased Random Walks (LERW) and time-inhomogeneous P\'olya urns. The use of LERW is called for by Wilson's algorithm to construct $\UST(n)$ (see \cite{Wilson1996} or \cite[Section 4.1]{LPbook}). The rooted $\UST(n,s)$ with an extra edge is obtained using a slight variation of Wilson's algorithm: we first add the extra edge $e=(x_0,x_1)$, then the first LERW path is run from $x_0$ to $x_1$. This procedure does not spoil the uniformity of measure because of the symmetries of Wilson's algorithm.

\medskip
\noindent \textbf{Colored Wilson's algorithm} for the construction of $\UST(n)_{+e}$: 
\begin{enumerate}

\item Enumerate the vertices as $\{x_0,x_1,\dots,x_{n-1}\}$, where $x_0=s$, and add the edge $(x_0,x_1)$. 

\item Run a LERW from $x_1$ until hitting $x_0$; the union of this walk with the existing edge will be denoted by $C_{-1}$.

\item Run a LERW from $x_2$ until hitting $C_{-1}$. This walk, including its endpoint, will be denoted by $R_0$, while $C_{-1}$ minus the endpoint of $R_0$ will be $B_0$. We also set $C_0=R_0\cup B_0$.

\item Now, recursively, given $C_i=R_i\cup B_i$ for some $i\geqslant 0$, run a LERW from $x_{i+3}$ until hitting $C_i$, denoting the walk, excluding its endpoint, by $\Delta_{i+1}$. If the endpoint is in $R_i$, then we let $R_{i+1}=R_i\cup \Delta_{i+1}$ and $B_{i+1}=B_i$; otherwise, we let $B_{i+1}=B_i\cup\Delta_{i+1}$ and $R_{i+1}=R_i$. 

\item Iterate this procedure until reaching $C_I = \{x_0,x_1,\dots,x_{n-1}\}$. In retrospect, set also $\Delta_{-1}=C_{-1}\setminus\{x_0\}$ and $\Delta_0 = C_0\setminus C_{-1}$.
\end{enumerate}

\begin{figure}[ht]
\centering
\begin{minipage}[b]{0.4\textwidth}
\centering
\includegraphics[scale=0.6]{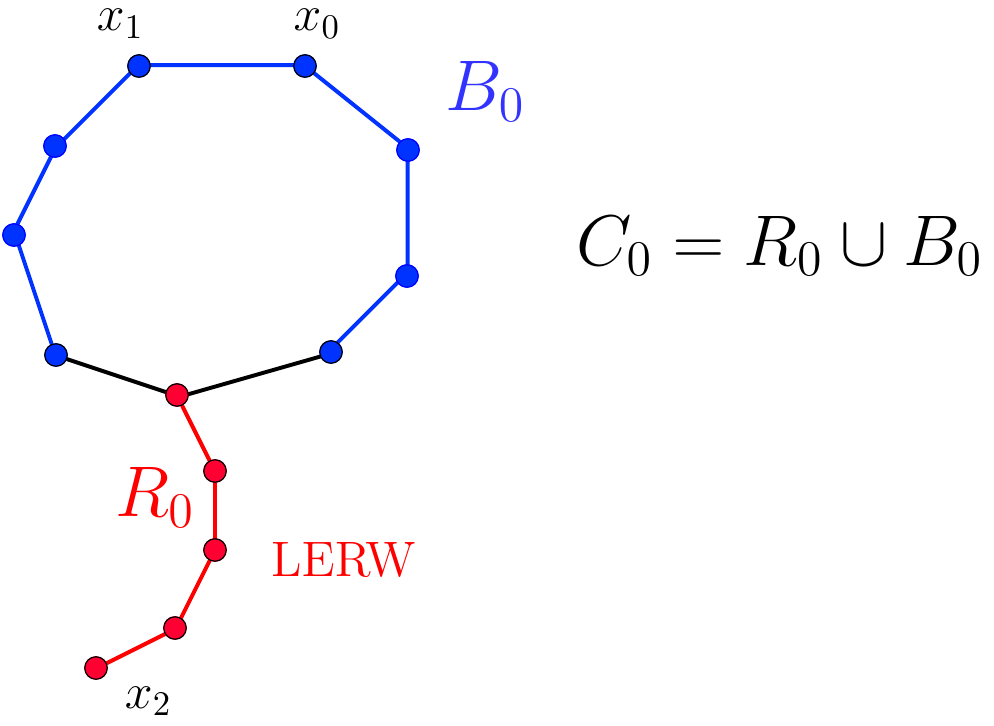}
\caption*{(a)}
\end{minipage}\vrule
\begin{minipage}[b]{0.59\textwidth}
\centering
\includegraphics[scale=0.6]{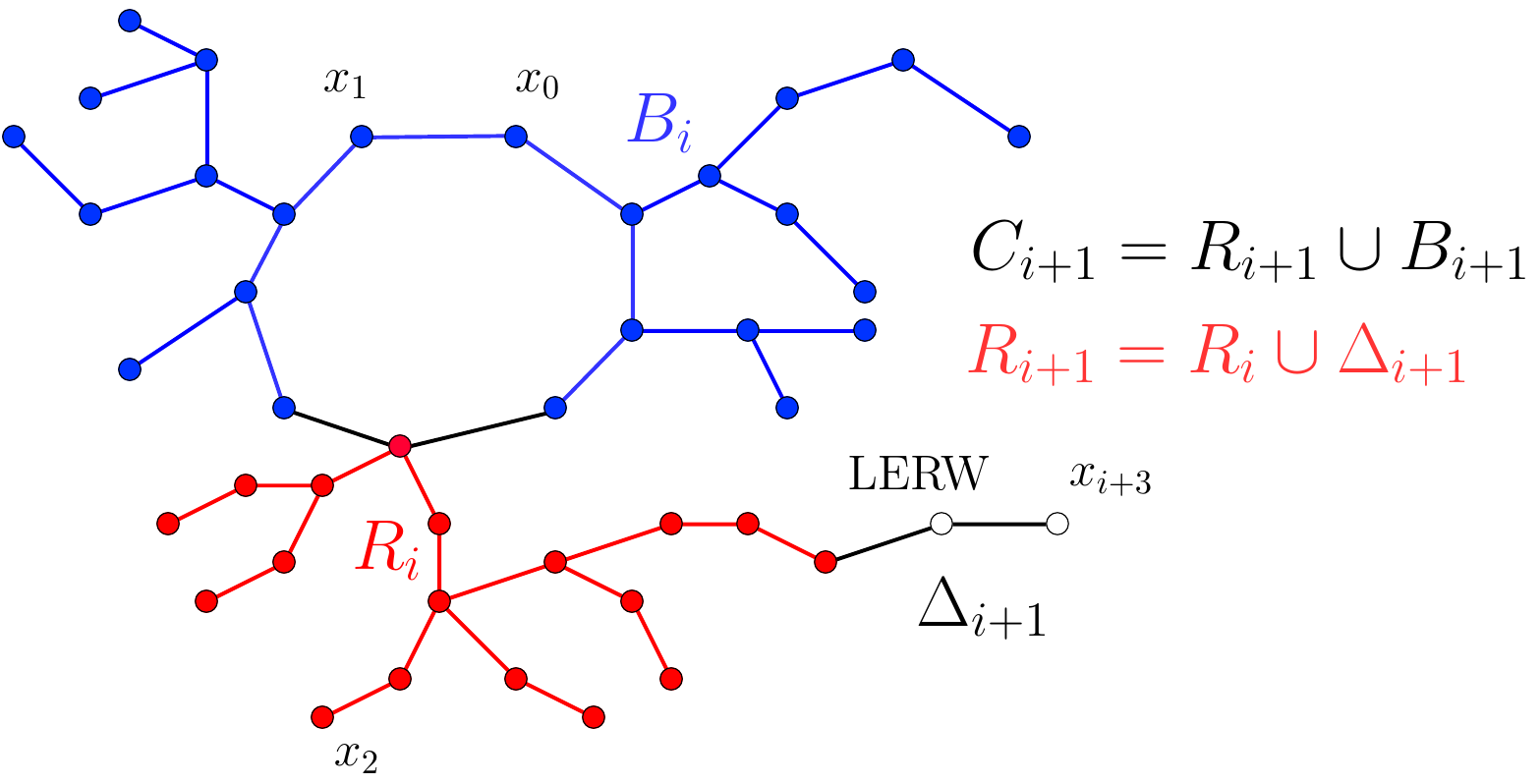}
\caption*{(b)}
\end{minipage}
\caption{Illustration of the colored Wilson's algorithm for constructing the rooted $\UST(n,s)$. (a) Step 3, with the emergence of two colored classes. (b) Step 4, the LERW from $x_{i+3}$ hitting $R_i$.}
\label{fig:Wilson_illustration}
\end{figure}

We will refer to $C_i$ as the set of colored vertices, red in $R_i$, blue in $B_i$. We are keeping track of the colors in order to understand equation \eqref{eq:kappa_tree_edge} for $\kappa(\UST(n),s)$ --- if the largest subtree is red, then $\kappa(\UST(n),s)=|B_I|$, otherwise we have $\kappa(\UST(n),s) \geqslant |R_I|$. Thus, our goal can be formulated as follows:

\begin{theorem}\label{PolyaEnd}
In the above framework of constructing a rooted $\UST(n)$, the final partition $\big(R_I,B_I\big)$ of the set $\{x_0,x_1,\dots,x_{n-1}\}$, as $n\to\infty$, is asymptotically almost surely non-trivial: for every $\varepsilon>0$ there exists $\delta>0$ such that 
$$
\PB{\frac{\kappa(\UST(n),s)}{n} > \delta} \geqslant \PB{|R_I| / |C_I| \in (\delta,1-\delta)} > 1-\varepsilon,
$$
for all $n>n_{\varepsilon}$.
\end{theorem} 

Note that $\Delta_{i+1}$ will be added to $R_i$ or $B_i$ with probability $|R_i|/\big(|R_i|+|B_i|\big)$ and $|B_i|/\big(|R_i|+|B_i|\big)$, respectively; moreover, the entire set $\Delta_{i+1}$ is independent of whether we add it to $R_i$ or $B_i$. Therefore, we can think of the process $\big(|R_i|,|B_i|\big)_{i\geqslant 0}$ as a P\'olya urn with random, symmetric, time-inhomogeneous increments: in each round $i$, we add $|\Delta_i|$ balls either to the red or to the blue balls, with probabilities proportional to the current color counts. In such a P\'olya urn process, when $(|\Delta_i|)_{i\geqslant 1}$ is uniformly bounded, the following result holds \cite{Pemantle1990}.

\begin{theorem}\label{thm.Pemantle}
Let the P\'{o}lya urn $\left(R_i,B_i\right)_{i\geqslant 0}$ has increments $\Delta_i$ bounded by some constant $M$ for all $i\geqslant 1$. Let $C_i=R_i+B_i$. Then the limiting distribution of $\lim\limits_{i\to\infty}\left(R_i/C_i, B_i/C_i\right)$ has no atoms in either coordinate.
\end{theorem}

\begin{corollary}\label{c.PemStop}
In the setting of the previous theorem, for any stopping time $I$ for the process $\left(R_i,B_i\right)_{i\geqslant 0}$, the distribution of $\left(R_I/C_I, B_I/C_I\right)$ has no atoms at 0 or 1. 
\end{corollary}

\begin{proof}[Proof of the corollary] By the strong Markov property of $\big(|R_i|,|B_i|\big)_{i\geq 0}$, we can continue the urn process from $\big(R_I,B_I\big)$, and if $|R_I| / |C_I|$ had an atom at $0$ or $1$, then it would survive in the limit law of $|R_i| / |C_i|$. 
\end{proof}

Of course, our increments $\Delta_i$ are not at all bounded; however, it will turn out that the sequence $\big(|\Delta_i|/\sqrt{n}\big)_{i\geq 1}^I$ is asymptotically almost surely bounded, i.e., for every $\eps>0$ there exists $M<\infty$ such that 
\begin{equation}\label{e.sqrtbdd}
\P\big[\Delta_i/\sqrt{n} < M \textrm{ for all } i=1,2,\dots,I\big]>1-\eps,
\end{equation}
for all $n>n_{\eps}$ large enough. Hence, with probability $1-\eps$, we can consider $\big(|R_i|,|B_i|\big)/\sqrt{n}$ as a symmetric time-inhomogeneous P\'olya urn process with bounded increments. Since the normalizations by $\sqrt{n}$ cancel each other out, Corollary~\ref{c.PemStop} implies that $|R_I| / |C_I|$ cannot have an atom at $0$ or $1$ of a uniformly positive size, as $n\to\infty$. Finally, we can take the mixture of these laws $|R_I| / |C_I|$ over the law of the initial configuration $\big(|R_0|,|B_0|\big)/\sqrt{n}$ and over the sequence $\big(\Delta_i/\sqrt{n}\big)_{i\geq 1}^I$, to obtain the final result.

\begin{proof}[Proof of \eqref{e.sqrtbdd}] First, we need to prove that the initial configuration $\big(|R_0|,|B_0|\big)/\sqrt{n}$ is non-trivial. It is a well-known result that the length of the first LERW path $C_{-1}$ has Rayleigh distribution in the limit \cite{camarri2000limit}. We present here the derivation taken from \cite{PeresRevelle2004}, since we will use a generalization of this argument. Pick two vertices $x$ and $y$ in the complete graph $K_n$, let $L=d_{\UST(n)}(x,y)$ be the length of the LERW path from $x$ to $y$, and denote by $\gamma_k$ the first $k$ vertices of the path itself. Because of the symmetries of $K_n$, we can write the following, with any specific sequence $x=x_0, x_1,\dots,x_{k-1}$ of distinct vertices that does not contain $y$: 
\begin{equation}\label{e.PF1}
\begin{split}
\Pb{L=k \md L\geqslant k} &= \Pb{L=k \md \gamma_k=\{x_0,\dots,x_{k-1}\}}\\
&=\PB{\Step(x_{k-1},y) \md \Past(x_0,\dots,x_{k-1}),\,\Future(x_0,\dots,x_{k-1})}\,;
\end{split}
\end{equation}
here, $\Past(x_0,\dots,x_{k-1})$ is the event that the SRW from $x$ to $y$ is currently at $x_{k-1}$ and its loop-erasure so far is $\{x_0,\dots,x_{k-1}\}$; the event $\Future(x_0,\dots,x_{k-1})$ is that the SRW continued from $x_{k-1}$ will not return to $\{x_0,\dots,x_{k-1}\}$ before hitting $y$; the event $\Step(x_{k-1},y)$ is that the first step of the SRW continued from $x_{k-1}$ will be to $y$; and finally, the equality between the two lines is by considering the last visit of the SRW to $x_{k-1}$. We can then continue~\eqref{e.PF1} as follows:
\begin{equation}\label{e.PF2}
\begin{split}
\frac{\Pb{\Step(x_{k-1},y) \md \Past(x_0,\dots,x_{k-1})}}{\Pb{\Future(x_0,\dots,x_{k-1}) \md \Past(x_0,\dots,x_{k-1}) } }= \frac{\frac{1}{n-1}}{\frac{1}{n-1}+\frac{n-k-1}{n-1}\frac{1}{k+1}}=\frac{k+1}{n}\,,
\end{split}
\end{equation}
where the second term in the denominator corresponds to first stepping to some vertex not in $\{x_0,\dots,x_{k-1},y\}$, then hitting $y$ before $\{x_0,\dots,x_{k-1}\}$. Now we can write a telescopic product:
\begin{equation}\label{e.prod}
\begin{split}
\Pb{L=k} &= \Pb{L=k \md L \geqslant k} \cdot \Pb{L > k-1 \md L\geqslant k-1} \cdots \Pb{L> 1 \md L\geqslant 1}\\
&= \frac{k+1}{n}\prod\limits_{j=1}^{k-1}\left(1-\frac{j+1}{n}\right).
\end{split}
\end{equation}

In order to understand the size of~\eqref{e.prod}, note that for all $0<x<1/2$ we have $\exp(-x-x^2)< 1-x < \exp(-x)$, and thus, for any $t>0$ and $n$ large enough,
$$
\exp\left(-\frac{\lfloor t\sqrt{n}\rfloor^2}{2n}-\frac{\lfloor t\sqrt{n}\rfloor^3}{3n^2}\right) < \prod_{j=1}^{\lfloor t\sqrt{n}\rfloor} \left(1-\frac{j+1}{n}\right) < \exp\left(-\frac{\lfloor t\sqrt{n}\rfloor^2}{2n}\right)\,.
$$
Therefore,~\eqref{e.prod} gives, as $n\to\infty$,
$$
\Pb{L=\lfloor t\sqrt{n}\rfloor} \sim \exp\left(-\frac{t^2}{2}\right)\frac{t}{\sqrt{n}}\,,
$$
which proves the Rayleigh limit law for $L/\sqrt{n}$.

In the further steps of the algorithm, the LERW hits the set $C_i$ instead of a single vertex, thus using similar arguments we can obtain the conditional distribution of the size of $\Delta_i$, for each $i\geqslant -1$:
\begin{equation}\label{DeltaRecur}
\Pb{|\Delta_{i+1}|=k \md |C_i|}= \frac{k+|C_i|}{n} \prod_{j=1}^{k-1} \left( 1-\frac{j+|C_i|}{n} \right).
\end{equation}
In particular, using the same calculus estimate as before, 
$$
\PB{R_0=\lfloor r\sqrt{n}\rfloor \md B_0=\lfloor b\sqrt{n}\rfloor } \sim \frac{r+b}{\sqrt{n}}  \exp\left(-\frac{(r+b)^2-b^2}{2}\right).
$$
That is, $(R_0,B_0)/\sqrt{n}$ indeed has an absolutely continuous limit law on $[0,\infty)^2$.

Continuing in the same way, assuming $|C_i|<n/4$ and $t\sqrt{n}<n/4$ in order for the calculus estimate to work, we have
$$
\PB{|\Delta_{i+1}| = \lfloor t\sqrt{n}\rfloor  \md |C_i|} \sim \left(\frac{t}{\sqrt{n}}+\frac{|C_i|}{n}\right) \exp\left(-\frac{t^2}{2}-\frac{t |C_i|}{\sqrt{n}}\right).
$$
Summing up, this implies for  $|C_i|<n/4$ that
\begin{equation}\label{eq.Delta_i}
\PB{|\Delta_{i+1}|> t\sqrt{n} \md |C_i|} \asymp \exp\left(-\frac{t^2}{2}-\frac{t|C_i|}{\sqrt{n}}\right).
\end{equation}
For $|C_i| \geqslant n/4$, equation~\eqref{DeltaRecur} implies an exponential tail for $|\Delta_{i+1}|$ directly. We will also use a slightly different version of~\eqref{eq.Delta_i}:
\begin{equation}\label{eq.boundondeltac}
\PB{|\Delta_{i+1}|\geqslant t\frac{n}{|C_i|} \md |C_i|}  \asymp \exp\left(-t-t^2\frac{n}{2|C_i|^2}\right).
\end{equation}

Now, in order to prove boundedness of $\big(|\Delta_i|/\sqrt{n}\big)_{i\geq 1}$, for all the steps till $I$, we need the probability $\Pb{|\Delta_{i+1}|> t\sqrt{n}}$ to decay fast as $i$ increases. In light of~\eqref{eq.Delta_i}, this is equivalent to $|C_i|$ increasing steadily. In fact, we will prove that there exist constants $a,A>0$ such that
\begin{equation}\label{e.CiSLLN}
\Pb{|C_i|^2 < ani} < \exp(-A i),
\end{equation}
where, moreover, $A\to\infty$ as $a\to 0$. Assuming this, using~\eqref{eq.Delta_i},
\begin{equation}\label{e.finalbound}
\begin{split}
\Pb{|\Delta_{i+1}|> t\sqrt{n}} &= \Pb{|\Delta_{i+1}|> t\sqrt{n}, |C_i|^2\geqslant ani}+ \Pb{|\Delta_{i+1}|> t\sqrt{n}, |C_i|^2< ani}\\
&< \PB{|\Delta_{i+1}|> t\sqrt{n} \md |C_i|^2\geqslant ani} + \Pb{|C_i|^2< ani}\\
&< C \exp(-t\sqrt{ai}) + \exp(-Ai)\,.
\end{split}
\end{equation}
Given any $\eps>0$, if we take first $A$ large (by taking a small $a$), then $t$ large, we get 
$$\sum_{i=1}^\infty \Big(C\exp(-t\sqrt{ai}) + \exp(-Ai)\Big) < \eps\,,$$ 
and~\eqref{e.sqrtbdd} is proved.

Now, in order to prove~\eqref{e.CiSLLN}, consider the increments 
\begin{equation}\label{eq.CiDi}
X_{i+1} := |C_{i+1}|^2-|C_i|^2=2\,|C_i|\,|\Delta_{i+1}|+|\Delta_{i+1}|^2\,,
\end{equation}
of the process $\big(|C_i|^2\big)_{i\geq 1}$. If we prove that, conditioned on an arbitrary $C_i$, there exist constants $b_1,b_2>0$ such that
\begin{equation}\label{e.Xibound}
\Pb{X_{i+1} < b_1n \md C_i} < b_2,
\end{equation}
with $b_2\to 0$ as $b_1\to 0$, then~\eqref{e.CiSLLN} will follow by a simple argument: (\ref{e.Xibound}) implies that the random variable $I:=\#\{1\leq j\leq i: X_{j+1} > b_1 n\}$ stochastically dominates a Binom$(i,1-b_2)$ variable, thus for any $A<\infty$, if $b_2>0$ is small enough, then $\Ps{I > i/2} \geq 1-\exp(-Ai)$, and on the event $\{I>i/2\}$ we have $|C_i|^2 > b_1 n i/2$, implying~\eqref{e.CiSLLN}.

To start the proof of~\eqref{e.Xibound}, we can assume that $C_i<n/4$, since otherwise $C_{i+1}\geqslant C_i+1$ implies $X_{i+1}\geqslant n/2$ deterministically. Now, if $|C_i|<\sqrt{n}$, then~\eqref{eq.Delta_i} implies that $\Pb{|\Delta_{i+1}|>b_1\sqrt{n} \md C_i}\to 1$ as $b_1\to 0$, and hence the term $|\Delta_{i+1}|^2$ in~\eqref{eq.CiDi} ensures that~\eqref{e.Xibound} holds. On the other hand, if $\sqrt{n}\leqslant |C_i|<n/4$, then~\eqref{eq.boundondeltac} implies that $\Pb{|C_i|\,|\Delta_{i+1}|>b_1\sqrt{n} \md C_i}\to 1$ as $b_1\to 0$, and the term $|C_i|\,|\Delta_{i+1}|$ in~\eqref{eq.CiDi} ensures that~\eqref{e.Xibound} holds. That is, in all cases we are done.
\end{proof}


\section{Critical versus near-critical Erd\H{o}s-R\'enyi graphs}\label{s.ER}

We will consider the Erd\H{o}s-R\'enyi graph $G(n,p)$ in its critical window for the emergence of a giant cluster, at $p=p_n(\lambda)=1/n+\lambda/n^{4/3}$, with $\lambda\in(-\infty,\infty)$. It is always helpful to use the standard monotone coupling of these random graphs: the edges of the complete graph $K_n$ have i.i.d.~$U_e\sim$Unif$[0,1]$ labels, and then the graph with edge set $\{e\in E(K_n): U_e\leqslant p\}$ is distributed like $G(n,p)$. This way, we can think of raising $p$ or $\lambda$ as adding edges to the graph.

It is well-known that the cluster of a typical vertex in $G(n,p_n(\lambda))$ in the critical case $\lambda=0$ is locally a GW tree with Poisson offspring distribution with mean $1$, while more globally, the largest clusters have size of order $n^{2/3}$ (see \cite[Section 12.3]{PGG} for an outline of the proof and references). As we are raising $\lambda$, extra edges appear in the standard coupling; since the number of possible edges in a cluster of size $n^{2/3}$ is of order $n^{4/3}$, the number of extra edges in each large component is approximately Poisson $(f(\lambda))$ for some function $f$, while, of course, the extra edges also merge some of these components. That is, the large scale structure of large critical versus near-critical clusters resembles but does not exactly coincide with our first example: a critical random tree conditioned to be large, plus a constant number of random edges.

One result making this picture more precise is that the probability that the largest cluster $\CC^n_1(\lambda)$ of $G(n,p_n(\lambda))$ is a tree converges to some $r_0(\lambda)>0$, which decays rapidly as $\lambda\to\infty$ (see \cite{JansonKnuth1993}). Moreover, conditioned on the largest cluster being a tree on $N$ vertices, it is clearly a uniform random tree $\UST(N)$. This means that, for $\lambda=0$, with a decent positive probability the SI spreading on $G$ will encounter bottlenecks everywhere during the process, and, because of these bottlenecks at random locations, the averaged spreading curve will not converge and will produce jumps. However, at large $\lambda>0$, a typical largest cluster will have one or more extra edges, hence, on a typical realization of the cluster, we expect to see the smoothing effect. The following results from our numerical simulations show that, even at a moderately off-critical value $\lambda=5$, the smoothing effect is typically quite severe (see Figure~\ref{fig:ER_simulation}).

\begin{figure}[ht]
\centering
\begin{minipage}[b]{0.49\textwidth}
\centering
\includegraphics[scale=0.45]{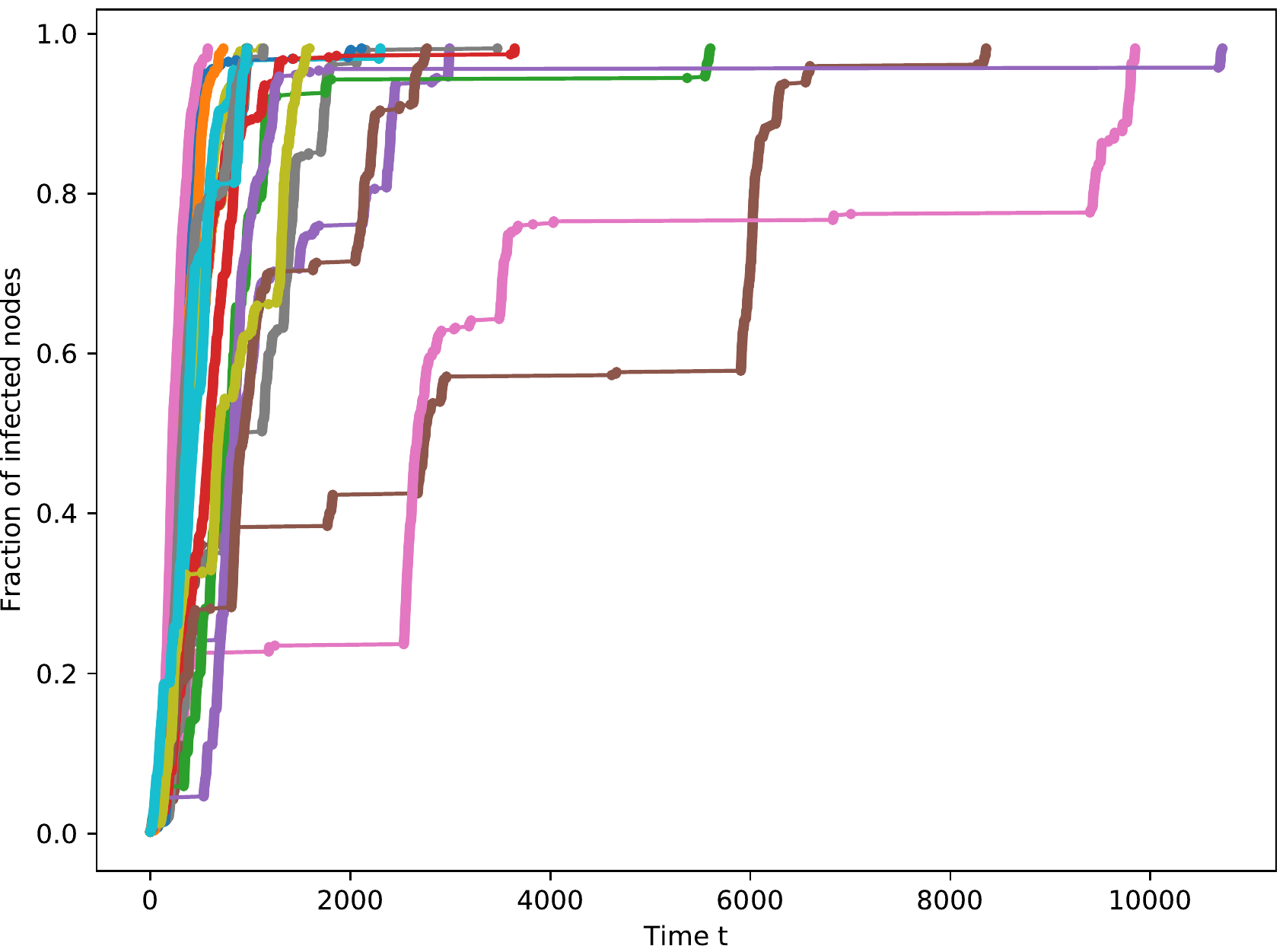}
\caption*{(a)}
\end{minipage}\hfill
\begin{minipage}[b]{0.49\textwidth}
\centering
\includegraphics[scale=0.45]{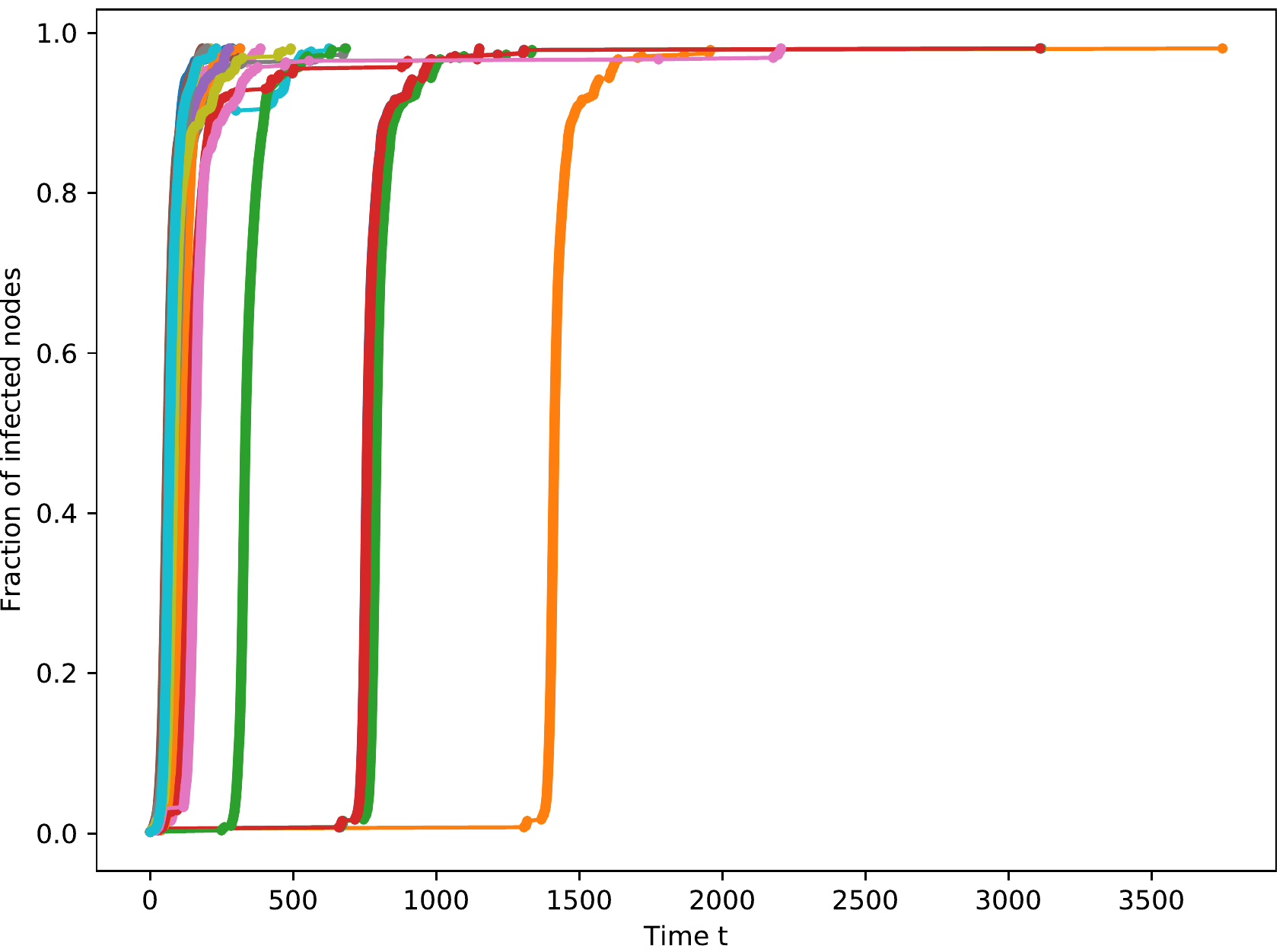}
\caption*{(b)}
\end{minipage}
\caption{Simulation of SI spreading with power law inter-event times with $\aa=0.8$ on the largest component of a near-critical Erd\H os-R\'enyi graph with $n$ vertices. The edge weights are fixed, 20 runs are shown with random starting vertices. (a) $\lambda=0$, $n=6000$ and the component has 541 vertices, no surplus edges. (b) $\lambda=5$, $n=1000$ the component has 515 vertices, 27 surplus edges.}
\label{fig:ER_simulation}
\end{figure}

Of course, in order to apply our Theorem~\ref{thm:main_smoothing}, we need to know the structure of $G$ with its extra edges, to prove that $\kappa(G,s)$ is typically positive when $\lambda$ is large. A seminal result of Aldous \cite{Aldous1997} says that the decreasingly ordered vector $\mathbf{C}^n(\lambda):=(|\CC^n_1(\lambda)|, |\CC^n_2(\lambda)|, \dots )$ of the sizes of the clusters of $G(n,p_n(\lambda))$, together with the vector $\mathbf{S^n}(\lambda):=(S^n_1(\lambda),S^n_2(\lambda),\dots)$ of the number of surplus edges in each $\CC^n_i(\lambda)$, meaning the number of edges additional to the minimum possible value $|\CC^n_i(\lambda)|-1$, has a joint scaling limit:
\begin{equation}\label{e.AldousCS}
\Big(\frac{\mathbf{C}^n(\lambda)}{n^{2/3}}, \mathbf{S}^n(\lambda)  \Big)\xrightarrow{\ d\ }  \Big(\mathbf{C}^\infty(\lambda), \mathbf{S}^\infty(\lambda)  \Big)\,,
\end{equation}
where the distribution of the limiting vector $\mathbf{C}^\infty(\lambda)$ is given by the excursion lengths (in decreasing order) of a one-dimensional Brownian motion with a parabolic drift, $B_t+\lambda t - t^2/2$, above its running minimum, and each surplus is given by the arrivals of a rate one planar Poisson point process inside the excursion. This implies, in particular, that there is a limiting distribution for the number of surplus edges in the largest cluster, $\Pb{S^n_1(\lambda)=k}\to \Pb{S^\infty_1(\lambda)=k} = r_k(\lambda)$, and it can be described using Brownian computations. For instance, for large $\lambda>0$, it is not hard to show that the Brownian perturbation around the parabola is unlikely to ruin completely the excursion of area $\Theta(\lambda^3)$ given by the positive part of the parabola, and therefore
\begin{equation}\label{e.largelambda}
\E S^\infty_1(\lambda) \asymp \lambda^3\quad\text{and}\quad S^\infty_1(\lambda)\xrightarrow{\P} \infty,\quad\text{as }\lambda\to\infty.
\end{equation}
For a proof via counting arguments, see~\cite{JansonKnuth1993}.

Building on the above \cite{Aldous1997}, the metric structure of large  near-critical clusters was described in \cite{Addario-Berry2010,Addario-Berry2010-2}. To start with, the reader should recall the notion of Aldous' {\it Brownian Continuum Random Tree} \cite{AldousCRTI,AldousCRTII,AldousCRTIII}, abbreviated as CRT from now on, which is a random real metric tree built from the Brownian Excursion measure, together with a mass measure that is supported on the leaves of the tree. It is the {\it Gromov-Hausdorff-Prokhorov limit} of critical GW trees with finite variance offspring distribution, conditioned to have volume $N$, edge lengths scaled by $\sqrt{N}$, with the mass measure coming from the uniform distribution on the vertices. (See, e.g., \cite{AldousCRTII,LeGall2005,abraham2013note} for the definitions and overviews.) Then it was proved in \cite{Addario-Berry2010} that the vector of metric spaces $(\CC^n_1(\lambda), \CC^n_2(\lambda), \dots )$ converges to some $(\CC^\infty_1(\lambda), \CC^\infty_2(\lambda), \dots )$ in the Gromov-Hausdorff topology extended to vectors in a suitable way. For the limit, the following construction was given in \cite[Procedure 1]{Addario-Berry2010-2}. 

\medskip
\noindent{\bf Construction.} Sample the sizes $(C_1,C_2,\dots)$ and surpluses $(S_1,S_2,\dots)$ for $(\CC^\infty_1(\lambda), \CC^\infty_2(\lambda), \dots )$ according to the limit distribution in~(\ref{e.AldousCS}). In particular, $S_1$ is given by the distribution $r_k(\lambda)$ mentioned before. Conditionally on these data, the coordinates $\CC^\infty_i(\lambda)$ will be independent from each other, and the distributions depend on the size $C_i$ only through a scaling. So, it is enough to describe $\CC^\infty_i(\lambda)$ conditionally on the scaled volume being $1$ and the surplus being some  $k\in\N$.
\begin{itemize}
\item If $k=0$ then let the component simply be a Brownian CRT of total mass 1. 
\item If $k=1$ then let $(X_1,X_2)$ be a $\mathrm{Dirichlet}(\frac 12,\frac 12)$ random vector, let $\mathcal{T}_1,\mathcal{T}_2$ be independent Brownian CRT's of sizes $X_1$ and $X_2$, and identify the root of $\mathcal{T}_1$ with a uniform leaf of $\mathcal T_1$ and with the root of $\mathcal T_2$, to make a ``lollipop'' shape. (For the definition of the Dirichlet distibutions, see \cite[Section~3.1]{Addario-Berry2010-2} or Wikipedia.)

\item If $k \geq 2$, let $K$ be a random 3-regular multigraph with $2(k-1)$ vertices chosen from a finite list, according to the following probability measure:
\begin{equation}\label{e.3reg}
\mu_k(K):=\frac{1}{Z_k} \Bigg(2^{t(K)}\prod_{e \in E(K)} \mathrm{mult}(e)!\Bigg)^{-1},
\end{equation}
where $t(K)$ is the number of loops in the multigraph $K$, $\mathrm{mult}(e)$ is the multiplicity of the edge $e$, and $Z_k$ is the normalization factor to get a  probability measure. This $K$ will be called the {\it kernel} of the cluster. Then:
\begin{enumerate}
\item Order the edges of $K$ arbitrarily as $e_1,\ldots,e_{3(k-1)}$, with $e_i=u_iv_i$. 
\item Let $(X_1,\ldots,X_{3(k-1)})$ be a $\mathrm{Dirichlet}(\frac 12,\ldots,\frac 12)$ random vector. 
\item Let $\mathcal{T}_1,\ldots,\mathcal{T}_{3(k-1)}$ be independent Brownian CRT's, with tree $\mathcal{T}_i$ having mass $X_i$, and for each $i$ let $r_i$ and $s_i$ be the root and a uniform leaf of $\mathcal T_i$. 
\item Form the component by replacing edge $u_iv_i$ with tree $\mathcal{T}_i$, identifying $r_i$ with $u_i$ and $s_i$ with $v_i$, for $i=1,\ldots,3(k-1)$.
\end{enumerate}
\end{itemize}

We will now pull back this information on the scaling limit to the discrete graphs:

\proof[Proof of Theorem~\ref{thm:ER}] {\bf (1)} As a warm-up, note that the weaker result 
\begin{equation}\label{e.kappa0}
\kappa(\CC_1^n,\sigma)/|\CC_1^n|\xrightarrow{\P}0
\end{equation} 
follows easily from the scaling limit description. It is well-known that the mass measure of the CRT is fully supported on the leaves; that is, a uniform random vertex in a discrete random tree on $N$ vertices that converges to the CRT will have, with probability tending to 1, only a single macroscopic path emanating from it, i.e., going to distance of order $\sqrt N$. By the mass measure having no atoms in the scaling limit, pieces of vanishing diameter have vanishing volume in the discrete tree, and hence by cutting the one macroscopic path, we can separate most of the tree from the uniform vertex. Together with the construction above that glues finitely many CRT's, we obtain~\eqref{e.kappa0}.

In order to prove the stronger statement, knowledge about the scaling limit will not suffice, since that cannot distinguish between different distances of size $o(\sqrt{|\CC_1^n|})$ and between different volumes of size $o(|\CC_1^n|)$. That is, we need to go back to the discrete model, and the breadth-first search exploration process used in \cite{Aldous1997}.

Let $\B$ be the event that the largest cluster is the cluster $\CC$ of vertex $1$. By symmetry, it is enough to prove that $\Pb{\kappa(\CC,1)>K \md \B}<\eps$ if $K>K_0(\eps)$ is large enough, for all $n>n_0(\lambda,\eps)$. Furthermore, let $\A_\delta$ be the event that the exploration process of the cluster of 1 finds that $\CC$ has size at least $\delta n^{2/3}$. Now observe the following two facts:
\begin{itemize}
\item For any $\delta>0$, there exists $\tilde\delta >  0$, tending to 0 with $\delta$, such that, for all large enough $n$,
$$
\Pb{\B \md \A_\delta} > \tilde\delta
\qquad\text{and}\qquad
\Pb{\A_\delta \md \B} > 1-\tilde\delta\,.
$$
These follow easily from \cite{Aldous1997}. Indeed, for the first inequality, the event $\{$all clusters other than $\CC$ are smaller than $\delta n^{2/3}\}$ has probability at least some $\tilde\delta$, it is positively correlated with $\A_\delta$, and their intersection implies $\B$. The second inequality follows from the size of the largest cluster having a scaling limit with no atom at 0. 
\item For the exploration process, given $\A_\delta$ with any $\delta>0$, the proof of Theorem~\ref{thm:GW_smoothing}~(1) applies, hence, for any $\eta>0$, if $K>K_0(\delta,\eta)$ is large enough, then 
$$
\Pb{\kappa(\CC,1)>K \md \A_\delta}<\eta\,.
$$
\end{itemize}
Now, combining these two facts, with $\eta=\tilde\delta^2$ and writing $\RR_K$ for the event $\{\kappa(\CC,1)>K\}$,
\begin{align*}
\Pb{\RR_K \md \B} = \frac{\Ps{\RR_K \cap  \B}}{\Ps{\B}} &\leq \frac{\Ps{\RR_K\cap\A_\delta} + \Ps{\B\setminus \A_\delta} }{\Ps{\B}}\\
&\leq \frac{\Ps{\RR_K\md \A_\delta} \Ps{\A_\delta}}{\Ps{\B\cap\A_\delta}} + \frac{\Ps{\B\setminus \A_\delta}}{\Ps{\B}}\\
& = \frac{\Ps{\RR_K \md \A_\delta}}{\Ps{\B\md\A_\delta}}+\Ps{\A_\delta^c \md \B} < \frac{\tilde\delta^2}{\tilde\delta} + \tilde\delta=2\tilde\delta\,.
\end{align*}
Since $\delta>0$ can be chosen so small that $2\tilde\delta<\eps$, this finishes the proof of part (1).
\medskip

\noindent {\bf (2)} For the lower bound, $\CC$ always contains a spanning tree with the structure of the $k=0$ case, converging to Aldous' CRT. In this limit, removing any vertex whose distance from the set of leaves is positive  breaks the tree into pieces of positive mass measure (by the self-similar nature of the CRT). Translating this to the discrete tree, of volume $N$, with probability close to 1, there exist vertices that have at least two disjoint paths of length $\delta \sqrt{N}$ if $\delta$ is small enough, and the resulting subtrees have volume $\delta_0 N$. 

For the upper bound, just notice that the surplus is at most some $k=k(\lambda,\eps)$ with probability at least $1-\eps$, then we break the total volume $N$ into at most $3k-3$ pieces, hence the largest of these pieces has volume at least $N / (3k-3)$. Regardless of how this piece fits in the kernel, by the structure of the CRT, with probability close to 1 if $\delta$ is small enough, it has a subtree of volume at least $\delta N / (3k-3)$ that can be separated from $\CC$ by a single edge. This implies the claimed  bound $(1-\delta_1(\lambda,\eps)) N$.
\medskip

\noindent {\bf (3)} The intuition for the proof is that the macroscopic structure of the largest cluster is more and more determined by the kernel as $\lambda\to\infty$, and this kernel more and more becomes an expander graph. This is similar to the strategy used in \cite{BenKozWor} to understand the mixing time of random walk on the giant cluster in a supercritical  Erd\H{o}s-R\'enyi graph. We will need two lemmas:

\begin{lemma}\label{lem.DirMax}
If $(X_1,\dots,X_n)$ is a $\mathrm{Dirichlet}(\frac 12,\ldots,\frac 12)$ random vector, then $\max_{1\leq i\leq n} X_i$ converges to 0 in probability, as $n\to\infty$.
\end{lemma}

\proof The marginal distribution of each $X_i$ is $\mathrm{Beta}(\frac12,\frac{n-1}{2})$. It is well-known and easy to calculate that $\Eb{ \mathrm{Beta}(\frac12,\frac{n-1}{2})^2} \sim \frac{3}{n^2}$. Thus, Markov's inequality for $X_i^2$ and a union bound give
$$
\PB{\max_{1\leq i\leq n} X_i > \frac{1}{n^{1/4}} } \leq n \, \Pb{X_i^2 > \frac{1}{n^{1/2}} } \leq n \frac{(3+o(1)) \sqrt{n}}{n^2} \to 0\,,
$$ 
as desired.\qed

\begin{lemma}\label{lem.3reg}
For any $\eps>0$ fixed, if $K$ is sampled according to $\mu_k$, then the probability that $K$ has a cut-edge that has at least $\eps k$ vertices on either side goes to 0 as $k\to\infty$.
\end{lemma}

\proof As explained, e.g., in \cite{Wormald1999models}, the weights in~(\ref{e.3reg}) mean that $K$ can be generated by the model of taking $3(k-1)$ independent random pairs on the set of $2(k-1)$ vertices, conditioned on 3-regularity. Then the results of \cite[Section 5]{BenKozWor} apply, meaning that $K$ is an expander with probability tending to 1, as  $k\to\infty$.\qed
\medskip

The combination of~\eqref{e.largelambda} with Lemma~\ref{lem.3reg} gives us that the kernel has no ``large-scale cut-edges'' with high probability as $\lambda\to\infty$. Combining also with Lemma~\ref{lem.DirMax}, we get that there is a 2-connected core of $\CC$ from which all subgraphs hanging off have small volume. This immediately implies the statement.
\qed


\section*{Acknowledgements}

We are indebted to J\'anos Kert\'esz for drawing our attention to the phenomenon studied in the paper, for many useful discussions, and constant support;  to J\'ulia Komj\'athy for suggesting that we look at near-critical Erd\H{o}s-R\'enyi graphs; to Bal\'azs R\'ath for many comments and discussions regarding the manuscript; and to Louigi Addario-Berry for some references.

Most of the work was done while AM was a PhD student at the Central European University, Budapest. AM also gratefully acknowledges financial support from the European Research Council grant ``Limits of discrete structures'', 617747, ARC (Federation Wallonia-Brussels) project "Big Data Models" and from Grant 16-01-00499 of the Russian Foundation of Basic Research. GP acknowledges support from the Hungarian National Research, Development, and Innovation Office, NKFIH grant K109684, and from an MTA R\'enyi Institute ``Lend\"ulet'' Research Group. 

\bibliographystyle{alpha}
\bibliography{references}

\end{document}